\documentclass[12pt,leqno]{amsart}
\usepackage{amssymb,latexsym}
\usepackage{mathrsfs}
\usepackage{amsrefs}

\usepackage[T1]{fontenc}
\usepackage[latin2]{inputenc}
\usepackage{amsfonts}
\usepackage{array}
\usepackage{threeparttable}
\usepackage[dvips]{graphicx}
\usepackage{rotating}
\usepackage{epsfig}
\usepackage[tableposition=top]{caption}

\newtheorem{thm}{Theorem}[section]
\newtheorem{pro}[thm]{Proposition}
\newtheorem{cor}[thm]{Corollary}

\theoremstyle{definition}
\newtheorem{rem}[thm]{Remark}
\newtheorem{exa}[thm]{Example}

\newtheorem*{xnta}{Notation}
\newtheorem*{xack}{Acknowledgments}

\numberwithin{equation}{section}

\begin{document}
\baselineskip=17pt

\title[]{Spectral properties of a class of random walks on locally finite groups}
\author[A. BENDIKOV, B. BOBIKAU, and CH. PITTET]{ALEXANDER BENDIKOV, BARBARA BOBIKAU, and CHRISTOPHE PITTET}
\address{A. Bendikov, Institute of Mathematics\\
 Wrocław University\\
50-384 Wrocław, Pl. Grunwaldzki 2/4, Poland,}
\email{bendikov@math.uni.wroc.pl} 
\address{B. Bobikau, Institute of Mathematics\\
 Wrocław University\\
50-384 Wrocław, Pl. Grunwaldzki 2/4, Poland,} 
\email{bobikau@math.uni.wroc.pl}
\address{Ch. Pittet, CMI, 39 rue Joliot-Curie, 13453 Marseille cedex 13, Universit\'{e} d'Aix-Marseille I, France,}
\email{pittet@cmi.univ-mrs.fr}

\date{}

\begin{abstract}
We study some spectral properties of random walks on infinite countable amenable groups
with an emphasis on locally finite groups, e.g. the infinite symmetric group $S_{\infty}$.
On locally finite groups, the random walks under consideration are driven by infinite divisible distributions. This allows us to embed our random walks into continuous time L$\acute{\textrm{e}}$vy processes whose heat kernels have shapes similar to the ones of $\alpha$-stable processes.
We obtain examples of fast/slow decays of return probabilities, a recurrence criterion, 
exact values and estimates of isospectral profiles and spectral distributions, formulae and estimates for the escape rates and
for heat kernels.

\end{abstract}
\renewcommand{\thefootnote}{}
\footnotetext{Research of A. Bendikov and B. Bobikau was
supported by the Polish Government Scientific Research Fund, Grant
N N201 371736. Ch. Pittet was supported by the CNRS.}
\subjclass[2000]{60-02, 60B15, 62E10, 43A05.}
\keywords{Random walk, locally
finite group, ultra-metric space, infinite divisible distribution, Laplace
transform, K$\ddot{\textrm{o}}$hlbecker transform, Legendre
transform, return probability, spectral distribution, isospectral profile, heat kernel.}

\maketitle

\tableofcontents
\renewcommand{\thefootnote}{\fnsymbol{footnote}}
\section{Introduction}

We apply methods from analysis (Laplace, K$\ddot{\textrm{o}}$hlbecker,  and Legendre transforms) and from geometric group theory (volume growth, iso\-pe\-ri\-metic and iso\-spectral inequalities)  to study the spectral properties of random walks on countable groups.
Most of the results are about groups which are \emph{not} finitely generated.
All the random walks we consider are symmetric and invariant under left translations.

\begin{xnta}
We use the following notation. For two non-negative functions defined on $\mathbb{R}_{+}$ or on a subset $I$ of $\mathbb{R}_{+}$ which is either a neighborhood of zero or of the infinity, we write:
\begin{itemize}
\item $f \asymp g$ if there are constants $ a_1,
a_2> 0$ such that \\ $\forall
x\in I$, $a_1 f(x) \leq g(x) \leq a_2 f(x)$,
(factor equivalence)
\item $f \stackrel{d}{\simeq} g$ if there are constants $b_1,
b_2> 0$ such that\\ $\forall
x\in I$, $f(b_1x) \leq g(x) \leq  f(b_2x)$,
(dilatational equivalence)
\item  $f \stackrel{d}{\asymp} g$ if there are constants $a_1,
a_2,b_1,b_2> 0$ such that \\$\forall
x\in I$, $a_1 f(b_1x) \leq g(x) \leq a_2
f(b_2x)$.
\end{itemize}
We also use the standard notation:
\begin{itemize}
\item $f \sim g$ at $a$ if $ f(x)/g(x)\to 1$ at $a$.
\end{itemize}

\end{xnta}

An early result by Kesten \cite{Kes} is 
that a countable group $G$ is amenable if and only if the 
spectral radius of the Markov operator associated to any symmetric irreducible 
random walk on $G$ equals $1$. In other words, \emph{a countable group is non-amenable if and only if the return probabilities of any symmetric irreducible random walk on it decay exponentially fast as time goes to infinity.}
(This equivalence generalizes to  locally compact groups \cite{Ber}.)
In Theorem \ref{cor1} , we prove that \emph{any locally compact unimodular non-compact (hence any countable infinite) amenable  group carries an irreducible symmetric random walk whose return probability decay is faster than any given sub-exponential function.} 

Recurrence criteria for random walks on countable infinitely generated groups have attracted much attention (see Chapter \ref{sec:rec} below for a short description of contributions by Brofferio and Woess \cite{Bro}, Darling and Erd\"os \cite{Dar}, Dudly \cite{Dud}, Flato and Pitt \cite{Fla}, Kasymdzhanova \cite{Kas}, Molchanov and Nabil \cite{Mol}, Revuz \cite{Rev}, Spitzer \cite{Spi}). Recall that a group is \emph{locally finite} if any of its finite subset generates a finite subgroup. Obviously, a group is countable locally finite  if and only if it is a countable increasing union of finite subgroups. Any such group is amenable. It is
not finitely generated if and only if the union is infinite and strictly increasing.  A typical example is the group $S_{\infty}$ of permutations of the integers with finite supports. Lawler \cite{Law} has obtained a sufficient condition for recurrence on countable locally finite groups (see Proposition \ref{pro7} below and  \cite[Proposition 1]{Bro}). In Proposition \ref{pro3.5}, we prove that \emph{Lawler's condition is also necessary for measures which are  convex linear combinations of idempotent measures.}

On an infinite finitely generated group, any symmetric irreducible random walk returns to the origin at time $t$ with probability at most $t^{-1/2}$ (up to a constant rescaling factor).
In other words, \emph{the slowest possible decay of return probabilities on an infinite finitely
generated groups, is the one of the simple random walk on $\mathbb Z$.} Let us define
the isospectral profile and indicate how to prove this fact (well known to the experts). 
Let $\mu$ be a symmetric irreducible probability measure on a countable group $G$
and let $P(f)=f*\mu$ be the corresponding  right convolution operator on $L^2(G)$.
The associated Laplace operator $\Delta=P-I$ is bounded, self-adjoint and $-\Delta$ is positive.
The \emph{isospectral profile} (also called the \emph{$L^2$-isoperimetric profile}) of $(G,\mu)$
is the function 
$$\Lambda(v)=\min_{|\Omega|\leq v}\lambda_1(\Omega),$$
where the minimum is taken over all finite subsets of $G$ of
cardinality less or equal to $v$. Here,
 $$\lambda_1(\Omega)=\min_{supp(f)\subset\Omega}\frac{(-\Delta f,f)}{\|f\|_2^2}.$$
We  first prove the above $t^{-1/2}$-bound in the case of the simple random walk on a Cayley graph of an infinite finitely generated group $G$. As $G$ is infinite, any finite subset of $G$ has at least one boundary point. A straightforward application of Cheeger's inequality \cite[Theorem 2.3]{Dod} implies that the isospectral profile of $G$ 
satisfies (up to a constant factor)
$$\Lambda(v)\geq v^{-2}, \forall v\geq 1.$$
The existence of a constant $C>1$, such that
$$\mu^{*t}(e)\leq Ct^{-1/2}, \forall t\geq 1,$$
then follows from \cite[Proposition V.1]{E}. The general case of a symmetric irreducible
probability measure $\nu$ on $G$ reduces to the case of a simple random walk $\mu$ because the irreducibility of $\nu$ obviously implies that we can choose $\mu$ with $supp(\mu)\subset supp(\nu)$.
Hence the Dirichlet form of $\nu$ 
$$(-\Delta_{\nu}f,f)=\frac{1}{2}\sum |f(x)-f(y)|^2\nu(x^{-1}y),$$
is bounded below  by the
Dirichlet form of $\mu$ times the non-zero constant factor 
$$\min_{g\in supp(\mu)}\nu(g)/\mu(g).$$
In contrast with finitely generated groups, \emph{there is no slowest decay of return probabilities on countable groups which are not finitely generated:} in Proposition \ref{cor3} we construct, on any infinite countable locally finite group, random walks with return probabilities whose decay is slower than any given positive function which goes to zero as time goes to infinity. 

On a  finitely generated group it is usually difficult (if not hopeless) to find an irreducible probability measure and to compute \emph{exact} values
of its isoperimetric profile, its spectral
density, or of the heat kernel driven by this measure. If $G$ is an infinite countable locally finite group,
we can choose finite subgroups $G_0\subset G_1\subset\cdots$ of $G$ such that 
$$G=\bigcup_{k\geq 0}G_k.$$ 
Any probability measure $\mu$ on $G$ can be represented as a convex linear combination
\begin{equation*}
\mu=\sum_{k=0}^{\infty} c_k \mu_k,
\end{equation*}
of probabilities $\mu_k$, each of which is supported by the finite subgroup $G_k$
(see Proposition \ref{pro3.1}).
A natural choice for $\mu_k$ is the homogeneous probability measure on $G_k$ (i.e. the normalized Haar measure of $G_k$): 
$$\mu_k=m_{G_k}.$$ 
The above representation of
$\mu$ as a series of Haar measures is convenient for computations because
if $\nu$ is a measure on a finite group $H$, then
$$\nu*m_H=m_H*\nu=m_H.$$

This in turn implies that \emph{$\mu$ is infinite divisible and  can be embed in a weakly continuous convolution semi-group $(\mu_t)$ of probability measures on $G$} (see Proposition \ref{proposit3.2}). It allows us
to compute explicitly a spectral resolution 
of the Laplace operator
$\Delta_{\mu}$ (Proposition \ref{pro4.1}), and to compute exact values of the isospectral 
profile $\Lambda_{\mu}$ (Theorem \ref{thm5.1} and Proposition \ref{proposit6.3}). 

Let $P$ be a  right convolution operator defined by a symmetric probability measure on a countable group.
Let $\lambda \to E_{\lambda}$ be the spectral resolution of the
Laplacian $-\Delta=I-P$, 
\begin{align*}
-\Delta=\int_0^{\infty} \lambda dE_{\lambda}.
\end{align*}
Let $\delta_e$ denote the characteristic function of the identity element $e\in G$. We define the spectral distribution function $\lambda \to
N(\lambda)$ as
\begin{align*}
N(\lambda):=(E_{\lambda} \delta_e, \delta_e).
\end{align*}
The asymptotic behavior of $N(\lambda)$, for $\lambda$ close to zero, for Laplace operators associated to simple random walks on 
finitely generated virtually nilpotent groups,
can be deduced from Varopoulos results  \cite[Lemma 2.46]{Lue}.
On finitely generated groups, and under some regularity assumption, \emph{the spectral distribution and the isospectral profile
are related by the formula}
$$N(\lambda)\stackrel{d}{\simeq}\frac{1}{\Lambda^{-1}(\lambda)},$$
(see \cite{BPSa}).
In the case of infinite countable locally finite groups (they are of course never finitely generated), if the decay of the coefficients
$c_k$ of the series of Haar measures $(\mu_k)$
$$\mu=\sum_{k=0}^{\infty} c_k \mu_k,$$
is not too fast,
namely, if  
there exists $\epsilon>0$, such that for all $n\in\mathbb N$,
	\[
		\sum_{k>n}c_k\geq\epsilon c_n,
	\]
then \emph{the same formula holds true} (see Proposition \ref{pro5} and Theorem \ref{pro4}).

In Section \ref{heat}, we define a  left-invariant ultra-metric $\rho$, associated to the random walk $\{X(n)\}$ driven by the measure $\mu$, whose balls are the subgroups $G_k$ and their left-cosets.  \emph{The ultra-metric $\rho$ and the computations of the functions $\Lambda(v)$ and $N(\lambda)$ (see Sec. 6 and 7) become crucial tools in estimating the rate of escape of the random walk $\{X(n)\}$ as well as its heat kernel/transition function,} see Proposition~\ref{pro8.4}, Proposition \ref{pro8.6}, Corollary \ref{cor8.7}, Example \ref{ex8.9} and Remark~\ref{rem8.10}.

\section{Convolution powers on unimodular groups.}

Let $(G,m)$ be a locally compact, non-compact, unimodular group endowed with a left Haar measure $m$.
Let $\{G_{k}\}_{k\in \mathbb{N}}$ be an increasing sequence of Borel subsets of $G$ such that
$|G_{k}|:=m(G_{k})\to \infty$ as $k\to \infty$.

With any sequence of positive reals $c=(c_{k})_{k\in \mathbb{N}}$ such that $\sum_{k\in \mathbb{N}}c_{k}=~1$ we associate the function
\begin{align*}
x \to M(x)=\sum_{k\in \mathbb{N}}\frac{c_{k}}{|G_{k}|} \textbf{1}_{G_{k}}(x).
\end{align*}

Evidently $B \to M(B) = \int_B M dm$ is a probability measure on~$G$. Assume that all $G_k$ are symmetric
and $\bigcup_k G_k$ generates a dense subgroup
of ${G}$. Then $M$ is symmetric and $supp \, M$ generates a dense subgroup of $G$.
Let $\mathfrak{L}_{M}: L^2 \to L^2$ be the
corresponding right-convolution operator $h\to h*M$. In general,
$\Vert\mathfrak{L}_{M}\Vert_{L^2\to L^2} \leq 1$ and it is equal
to $1$ if and only if the group ${G}$ is amenable. On the other
hand, let $\{X_k\}$ be i.i.d. on ${G}$ with law
$\mathbb{P}_{X_1}=M$ and let $S_n=X_1\cdot X_2 \cdot \dots~\cdot
X_n$ be the corresponding random walk on ${G}$. According to
\cite{Ber} the following characterization of $S_n$ via the norm of
the convolution operator $\mathfrak{L}_{M}$ holds: For all
relatively compact neighborhoods $V$ of the neutral element $e\in
{G}$,
\begin{align*}
\lim_{n \to \infty} \mathbb{P}(S_{2n}\in
V)^{1/2n}=\Vert\mathfrak{L}_{M}\Vert_{L^2\to L^2}.
\end{align*}
In particular, if ${G}$ is amenable
$\Vert\mathfrak{L}_{M}\Vert_{L^2\to L^2}=1$ and therefore
\begin{align*}
\mathbb{P}(S_{2n}\in V)=\exp(-n\cdot o(1)) \qquad \textrm{as}
\quad n \to \infty.
\end{align*}
If the group ${G}$ is not amenable, then
$\Vert\mathfrak{L}_{M}\Vert_{L^2\to L^2} <1$. This implies that
the decay at infinity of the function $n\to \mathbb{P}(S_{2n} \in
V)$ is always exponential.

\textit{We claim that for any non-compact unimodular  amenable
group ${G}$, the decay of the function $n\to \mathbb{P}(S_{2n} \in
V)$ can be made as close as possible to the exponential one by an
appropriate choice of the probability measure
$M =\mathbb{P}_{X_1}$.} 

Observe that $M^{*n}(V) \leq \Vert M^{*n}\Vert_{\infty} m(V)$, hence to prove our claim it is enough to estimate the decay of the function $n \to \Vert M^{*n}\Vert_{\infty} $.
We denote $\sigma_{k}=\sum_{i \leq k} c_{i}$ and
$\sigma(k)=1-\sigma_{k}$. Let $\lambda \to N(\lambda)$ be a
right-continuous, non-decreasing step-function $ \mathbb{R} \to \mathbb{R}_{+} $ having jumps at
the points $\lambda_{k}=\sigma(k)$ and taking values at these
points $N(\lambda_{k})=1/|G_{k}|$. Notice that $\lambda \to
N(\lambda)$ must be continuous at $\lambda=0$ and $N(\lambda)=0$
for $\lambda \leq 0$. Following an idea of Saloff-Coste \cite{Sal}, we prove the following statement. 
\begin{pro} \label{pro0}
In the notation introduced above the following inequality holds:
\begin{align*}
\parallel M^{*n} \parallel_{\infty} \leq \int_{\mathbb{R}_{+}} e^{-n\sigma}dN(\sigma), \quad n \in \mathbb{N}.
\end{align*}
\end{pro}
\begin{proof}
Observe that, since $G$ is unimodular, for any two functions $f$ and $g$ we have
\begin{align*}
\parallel f*g \parallel_{\infty} \leq \parallel f \parallel_{1} \parallel g\parallel_{\infty}, \quad \parallel f*g\parallel_{\infty}\leq \parallel f\parallel_{\infty} \parallel g\parallel_{1}.
\end{align*}
Hence,
\begin{align*}
\parallel M^{*n} \parallel_{\infty} \leq &
\sum_{k_1, \ldots, k_n} c_{k_{1}} \cdot \ldots \cdot   c_{k_{n}} \parallel \frac{1}{| G_{k_{1}}|}\textbf{1}_{G_{k_{1}}}* \ldots * \frac{1}{| G_{k_{n}}|}\textbf{1}_{G_{k_{n}}}\parallel_{\infty} \leq\\
\leq &
\sum_{k_1, \ldots, k_n} c_{k_{1}} \cdot \ldots \cdot   c_{k_{n}}  \min \left\{\frac{1}{| G_{k_{1}}|}, \ldots, \frac{1}{| G_{k_{n}}|}\right\}=\\
=&\sum_{k_1, \ldots, k_n} c_{k_{1}}\cdot \ldots \cdot   c_{k_{n}} \frac{1}{|G_{\max \{k_1, \ldots, k_n \}}|}=
\end{align*}
\begin{align*}
=&\sum_{k=0}^{\infty}  \frac{1}{|G_k|} \sum_{\max \{k_1, \ldots, k_n \}=k} c_{k_{1}}\cdot \ldots \cdot   c_{k_{n}}=\frac{c_0^n}{|G_0|}+\\
+&\sum_{k=1}^{\infty}  \frac{1}{|G_k|} \left( \sum_{k_1, \ldots, k_n \leq k}c_{k_{1}}\cdot \ldots \cdot   c_{k_{n}} -   \sum_{k_1, \ldots, k_n \leq k-1}c_{k_{1}}\cdot \ldots \cdot   c_{k_{n}}   \right)=\\
=&\frac{c_0^n}{|G_0|}+\sum_{k=1}^{\infty}  \frac{1}{|G_k|} \left( (c_0 + \ldots + c_{k})^n-(c_0 + \ldots + c_{k-1})^n \right)=\\
=&\sum_{k=0}^{\infty} \frac{\sigma_{k}^{n}}{|G_k|}- \sum_{k=1}^{\infty} \frac{\sigma_{k-1}^{n}}{|G_k|}=\sum_{k=0}^{\infty}\sigma_{k}^{n} \left(\frac{1}{|G_k|}-\frac{1}{|G_{k+1}|} \right)=
\end{align*}
\begin{align*}
=&\sum_{k=0}^{\infty} \left( 1-\sigma(k)\right)^n \left(\frac{1}{|G_k|}-\frac{1}{|G_{k+1}|} \right)=\int_{\mathbb{R}_{+}} \left( 1-\sigma \right)^n dN(\sigma) \leq \\
\leq & \int_{\mathbb{R}_{+}} e^{-n \sigma}dN(\sigma) .
\end{align*}
The proof is finished.
\end{proof}

We write $N=e^{-\mathcal{M}}$ and introduce the following  auxiliary transforms.
\begin{itemize}
\item The K$\ddot{\textrm{o}}$hlbecker transform of $\mathcal{M}$:
\begin{align*}
\mathcal{K}(\mathcal{M})(x):=-\log\left(
\int_0^{\infty}e^{-xt}de^{-\mathcal{M}(t)}\right).
\end{align*}
\item The Legendre transform of $\mathcal{M}$:
\begin{align*}
\mathcal{L}(\mathcal{M})(x):= \inf_{\tau > 0}\{x\tau +
\mathcal{M}(\tau)\}.
\end{align*}
\end{itemize}
Clearly, $\mathcal{L}(\mathcal{M}): \mathbb{R}_{+} \to \mathbb{R}_{+}$ is a non-decreasing, concave function. In particular, it is continuous. For any non-decreasing, continuous function $\mathcal{F}: \mathbb{R}_{+} \to \mathbb{R}_{+}$, we define the conjugate Legendre transform of~$\mathcal{F}$ as follows:
\begin{align*}
\mathcal{L}^*(\mathcal{F})(x):=\sup_{\tau > 0}\{-x\tau +
\mathcal{F}(\tau)\}.
\end{align*}

\begin{pro} \label{thm1} With the above notation, the following properties hold true.
\begin{itemize}
\item[1.] $\mathcal{K}(\mathcal{M})(x) \sim \mathcal{L}(\mathcal{M})(x)$ at $\infty$.
\item[2.] $\mathcal{L}(\mathcal{L}^*(\mathcal{F})) \geq \mathcal{F}$, and $\mathcal{L}(\mathcal{L}^*(\mathcal{F}))=\mathcal{F}$ if $\mathcal{F}$ is concave.
\item[3.] For any continuous decreasing function $\mathcal{M}:  \mathbb{R}_{+} \to \mathbb{R}_{+}$ such that $\mathcal{M}(+0)=+\infty$,
    \begin{align*}
    \mathcal{L}(\mathcal{M}) \asymp \mathcal{M}\circ(\mathcal{M}/id)^{-1}.
    \end{align*}
\item[4.] For any continuous function $\mathcal{F}: \mathbb{R}_{+} \to \mathbb{R}_{+}$ such that $\mathcal{F}(+\infty)=+\infty$ and $\mathcal{F}(t)=o(t)$ at $\infty$, there exist two continuous decreasing functions $\mathcal{M}_1$, $\mathcal{M}_2:\mathbb{R}_{+} \to \mathbb{R}_{+}$, with $\mathcal{M}_i(+0)=+\infty$, $i=1,2$, such that
      \begin{align*}
    \mathcal{L}(\mathcal{M}_1)/ \mathcal{F} \to \infty \quad \textrm{ and } \quad  \mathcal{L}(\mathcal{M}_2)/ \mathcal{F} \to 0 \quad \textrm{ at } \quad \infty.
    \end{align*}
\end{itemize}
\end{pro}
\begin{proof}
The first statement is Lemma 3.2 in \cite{Ben} (see also \cite{BPSa}).
For completeness we give a short proof of this result: For fixed
$x>0$ consider a positive function $m_x(t)=xt+\mathcal{M}(t)$ on
$(0, \infty)$. The function $m_x$ tends to $\infty$ at $0$ and at
$\infty$. Let $t_x$ be the smallest $t$ at which $m_x$ almost attains its
infimum, so that $(1+\epsilon)\mathcal{L}(\mathcal{M})(x)\geq m_x(t_x)$. Observe
that $\mathcal{M}$ is a decreasing function such that
$\mathcal{M}(+0)=+\infty$. We have
\begin{align*}
\int_0^{\infty} e^{-xt}de^{-\mathcal{M}(t)}&=x\int_0^{\infty}e^{-(xt+\mathcal{M}(t))}dt
\geq x\int_{t_x}^{\infty} e^{-(xt+\mathcal{M}(t))}dt\\
&\geq x e^{-\mathcal{M}(t_x)}
\int_{t_x}^{\infty}e^{-xt}dt=e^{-(xt_x+\mathcal{M}(t_x))}\geq e^{-(1+\epsilon)\mathcal{L}(\mathcal{M})(x)}.
\end{align*}
This proves the desired lower bound. For the upper bound, write
\begin{align*}
\int_0^{\infty} e^{-xt}de^{-\mathcal{M}(t)}&=x\int_0^{\infty}e^{-(xt+\mathcal{M}(t))}dt\\
&\leq x \left(\int_0^{\mathcal{L}(\mathcal{M})(x)/x} e^{-(xt+\mathcal{M}(t))}dt+ \int_{\mathcal{L}(\mathcal{M})(x)/x}^{\infty}e^{-xt}dt\right)\\
&\leq  x \int_0^{\mathcal{L}(\mathcal{M})(x)/x} e^{-\mathcal{L}(\mathcal{M})(x)}dt +\int_{\mathcal{L}(\mathcal{M})(x)}^{\infty} e^{-u}du\\
 &=\mathcal{L}(\mathcal{M})(x)e^{-\mathcal{L}(\mathcal{M})(x)}+e^{-\mathcal{L}(\mathcal{M})(x)}.
\end{align*}
That $\mathcal{K}(\mathcal{M})\sim \mathcal{L}(\mathcal{M})$ at
infinity follows easily from these two bounds.

The second statement is a standard fact from convex analysis, see for instance \cite{Roc} or \cite{Dyn}.

Proof of the third statement: Let $\mathcal{L}(\mathcal{M}):=J$. Then for all $t,x >0 $:
 \begin{align*}
J(t) \leq tx +  \mathcal{M}(x).
 \end{align*}
Let $x_{*}$ be such that $tx_{*}= \mathcal{M}(x_{*})$, then $x_{*}=(\frac{M}{id})^{-1}(t)$ and
\begin{align}\label{2.1}
J(t) \leq tx_{*} +  \mathcal{M}(x_{*})=2 \mathcal{M}(x_{*})=2 \mathcal{M}\circ \left(\frac{M}{id}\right)^{-1}(t).
 \end{align}
On the other hand, there exists $x^{*}$ such that $J(t)=tx^{*} +  \mathcal{M}(x^{*})$. Then $x^{*} \leq \frac{J(t)}{t}$ and therefore
\begin{align*}
 \mathcal{M}\left(\frac{J(t)}{t}\right) \leq \mathcal{M}(x^{*})\leq J(t)= \frac{tJ(t)}{t},
 \end{align*}
hence,
\begin{align*}
\left(\frac{ \mathcal{M}}{id}\right) \left(\frac{J(t)}{t}\right) \leq t.
 \end{align*}
Since $t \to (\mathcal{M}/id)(t)$ is continuous and strictly decreasing
\begin{align*}
\frac{J(t)}{t} \geq  \left(\frac{\mathcal{M}}{id}\right)^{-1}(t),
 \end{align*}
or equivalently
\begin{align}\label{2.2}
J(t) \geq t  \left(\frac{\mathcal{M}}{id}\right)^{-1}(t)= \mathcal{M} \circ \left(\frac{\mathcal{M}}{id}\right)^{-1}(t).
 \end{align}
Finally, the inequalities (\ref{2.1}) and (\ref{2.2}) give the desired result.

We prove statement 4:
 (a) Choose a continuous  $ \widetilde{\mathcal{F}_1}:\mathbb{R}_{+} \to \mathbb{R}_{+}$ such that close to zero $\widetilde{\mathcal{F}_1}>c>0$,  $\widetilde{\mathcal{F}_1}/\mathcal{F} \to \infty$ and   $\widetilde{\mathcal{F}_1}(\tau)=o(\tau)$ at $\infty$. Put $\mathcal{M}_1=\mathcal{L}^{*}(\widetilde{\mathcal{F}_1})$ and observe that $\mathcal{M}_1$ is convex, hence continuous.
The second statement of Proposition \ref{thm1} implies that
\begin{align*}
\mathcal{L}(\mathcal{M}_1)/ \mathcal{\mathcal{F}} \geq \widetilde{\mathcal{F}_1}/\mathcal{\mathcal{F}}\to \infty \quad \textrm{ at } \infty.
\end{align*}
(b) Choose a concave $ \widetilde{\mathcal{F}_2}:\mathbb{R}_{+} \to \mathbb{R}_{+}$
 such that near zero $\widetilde{\mathcal{F}_2}(t)\geq c>0$,  $\widetilde{\mathcal{F}_2}(t)/\mathcal{F}(t) \to 0$ and
$\widetilde{\mathcal{F}_2}(\tau)=o(\tau)$ at $\infty$.

Put  $\mathcal{M}_2=\mathcal{L}^{*}(\widetilde{\mathcal{F}_2})$. The function  $\mathcal{M}_2$ has the desired properties and according to the second statement of Proposition \ref{thm1},
\begin{align*}
\mathcal{L}(\mathcal{M}_2)/ \mathcal{\mathcal{F}} = \widetilde{\mathcal{F}_2}/\mathcal{\mathcal{F}}\to 0 \quad \textrm{ at }\quad \infty.
\end{align*}
The proof is finished.
\end{proof}

Some particular results based on the direct computation of
$\mathcal{L}(\mathcal{M})$ are presented in Table 1 (compare with  Property 3  of Proposition~\ref{thm1}).
\begin{table}[h]
\begin{center}
\begin{tabular}{|c|c|c|c|}
\hline &1 & 2 & 3\\
\hline \hline $\mathcal{M}(s)  \,$ at $0$&
$(\log\frac{1}{s})^{\alpha}$, \rule[-0.3cm]{0pt}{0.9cm}$\alpha>0$&
$s^{-\beta}$,
$\beta>0$ &$\exp_{(k)}\{s^{-\nu}\}$, $\nu>0$ $^{(*)}$\\
\hline $\mathcal{L}(\mathcal{M})(t)  \,$ at $\infty$&
\rule[-0.3cm]{0pt}{0.9cm}$(\log
t)^{\alpha}$ &$c_{\beta}t^{\beta_0}$ $^{(***)}$ &$\frac{t}{(\log _{(k)}t)^{\frac{1}{\nu}}}$ $^{(**)}$ \\

\hline
\end{tabular}
\scriptsize{$(*)$
$\rule[-0.6cm]{0pt}{1cm}\exp_{(k)}(t)=\underbrace{\exp\left(\exp\left(...\exp(t)\right)\right)}_{\textrm{k
times}}$, $(**)$ $\log_{(k)}(t)=\underbrace{\log
\left(1+\log\left(1+...\log(1+t)\right)\right)}_{\textrm{k
times}}$.\\
$(***)\rule[-0.4cm]{0pt}{0.8cm}$ $c_{\beta}=(1+\beta)/\beta^{\beta_0}$, $\quad \beta_0=\beta/(1+\beta) \qquad \qquad \qquad \qquad \qquad \qquad \qquad \qquad \qquad \quad$} \caption{Some examples.}
\end{center}
\end{table}

Let us show, for instance, how to compute the Legendre transform of the
function $\mathcal{M}: \tau \to \exp_{(k)}\{\tau^{-\nu}\}$, for
$k>1$ and $\nu>0$. Denote $R(\tau):=t \tau +\mathcal{M}(\tau)$.
The function $R(\tau)$ is strictly convex and tends to $\infty$ at
$0$ and at $\infty$. Let $\tau_{*}$ be the (unique!) value of
$\tau$ at which $R(\tau)$ attains its minimum, so that
$R(\tau_{*})=\mathcal{L}(\mathcal{M})(t)$. Since $\tau \to
R(\tau)$ is smooth, we obtain the following equation
\begin{align*}
0=& R'(\tau_{*})=t+\mathcal{M}'(\tau_{*})=\\
=& t-\frac{\nu}{\tau_{*}^{\nu +1}}\mathcal{M}(\tau_{*})\log
\mathcal{M}(\tau_{*}) \log_{(2)} \mathcal{M}(\tau_{*}) \dots
\log_{(k-1)} \mathcal{M}(\tau_{*}),
\end{align*}
which, in turn, implies the following two crucial properties.\\
$(1) \quad \log_{(k)} t \sim \tau_{*}^{-\nu} \, \textrm{ as } \, t
\to \infty. \quad \textrm{ In particular,} \quad \tau_{*}
\to 0 \quad \textrm{ at } \quad \infty.$
\begin{align*}
(2) \quad
\frac{\mathcal{M}(\tau_{*})}{\tau_{*}t}=\frac{\tau_{*}^{\nu}}{\nu
\log \mathcal{M}(\tau_{*}) \log \log \mathcal{M}(\tau_{*}) \dots
\log_{(k-1)} \mathcal{M}(\tau_{*})} \to 0 \, \textrm{ at } \, \infty.
\end{align*}
Finally, thanks to $(1)-(2)$ we arrive to the desired conclusion
\begin{align*}
\mathcal{L}(\mathcal{M})(t)=R(\tau_{*})=t \tau_{*}\left(1+
\frac{\mathcal{M}(\tau_{*})}{\tau_{*}t}\right) \sim t \tau_{*}
\sim \frac{t}{(\log_{(k)}t)^{1/ \nu}}, \quad \textrm{at} \quad \infty.
\end{align*}
\begin{rem} \label{rem2.3} The same method works also in a slightly more general setting:
Let $r: \mathbb{R}_{+} \to \mathbb{R}_{+}$ be a strictly
increasing function such that $r(+\infty)=+\infty$. Assume that $\lambda
r'(\lambda) \asymp r(\lambda)$ at $\infty$. Let
$\mathcal{M}(\tau)=\exp_{(k)}\left(r(1/\tau)\right)$, defined for $\tau >0$.
Then,
\begin{align*}
\mathcal{L}(\mathcal{M})(t)\asymp \frac{t}{r^{-1}(\log_{(k)}(t))}
\qquad \textrm{at} \quad \infty.
\end{align*}
\end{rem}
\begin{thm}\label{cor1}
Let $G$ be a locally compact non-compact group. Assume that $G$ is both unimodular and amenable. Let $\mathcal{F}: \mathbb{R}_{+} \to \mathbb{R}_{+}$ be a non-decreasing function such that $\mathcal{F}(t)=o(t)$ at infinity. There exists a symmetric strictly positive probability density $M$ on $G$ such that
\begin{align*}
-\log \parallel M^{*n}\parallel_{\infty}/\mathcal{F}(n) \to \infty \quad \textrm{at} \quad \infty.
\end{align*}
\end{thm}
\begin{proof}
For $\mathcal{F}$ given,  choose $\mathcal{M}$ such that $\mathcal{L}(\mathcal{M})/\mathcal{F}\to \infty$ at $\infty$. See Proposition \ref{thm1} (4). Choose an increasing sequence of Borel sets $G_k \subset G$ such that $|G_k|\to \infty$ and define a decreasing sequence $\{\sigma(k)\}$ from the equation $e^{-\mathcal{M}(\sigma(k))}=|G_{k}|^{-1}$, $k>>1$. Then define a sequence $\{c_k\}$ as follows: For $k \geq k_0>>1$ put $c_k=\sigma(k-1)-\sigma(k)$ and for $1 \leq k< k_0$ choose $c_k>0$ such that $\sum_{k=0}^{\infty}c_k=1$. Finally define
\begin{align} \label{equat2.3}
M=\sum_{k=0}^{\infty}c_{k}\frac{1}{|G_{k}|} \textbf{1}_{G_{k}},
\end{align}
and a step-function $N$ which is right-continuous, non-decreasing and has jumps at points $\lambda_k=\sigma(k)$ with values $N(\lambda_k)=1/|G_{k}|$. Write $N=e^{-\widetilde{\mathcal{M}}}$. Clearly $\widetilde{\mathcal{M}}\geq \mathcal{M}$ and therefore $\mathcal{K}(\widetilde{\mathcal{M}})\geq \mathcal{K}(\mathcal{M})$.
Applying Proposition $\ref{pro0}$ we come to the desired conclusion
\begin{align*}
 & -\log \parallel M^{*n}\parallel_{\infty}/\mathcal{F}(n) =  \mathcal{K}\left(\widetilde{\mathcal{M}}\right)/\mathcal{F}(n)\geq \\
& \geq
\mathcal{K}(\mathcal{M})/\mathcal{F}(n)\sim \mathcal{L}(\mathcal{M})/\mathcal{F}(n) \to \infty \qquad \textrm{ at }\quad \infty.
\end{align*}
The proof is finished.
\end{proof}
This proves the claim stated in the beginning of this section. Observe that the construction given in the proof of Theorem \ref{cor1}, Table 1 and Remark \ref{rem2.3} can be used to build many examples of fast decaying convolution powers on locally compact non-compact unimodular amenable groups. For example, choose $l \geq 1$ and for $k \geq k_0$ big enough put 
\begin{align*}
\sigma(k)=\left( \log_{(l+1)}|G_k|\right)^{\frac{1}{\nu}}.
\end{align*}
Define $c_k>0$ such that $\sum_{k \geq 0}c_k=1$ and $c_k=\sigma(k-1)-\sigma(k)$ for $k \geq k_0+1$. Then the probability density $M$ defined by (\ref{equat2.3}) satisfies
\begin{align*}
\parallel M^{*n} \parallel_{\infty} \leq \exp \left( -\frac{cn}{(\log_{(l)}n)^{\frac{1}{\nu}}}\right) \qquad \textrm{at} \quad \infty,
\end{align*}
for some constant $c>0$.

\section{Convolution powers on locally finite groups.}

Assume that the group $G$ under consideration has the following structure: there exists a strictly increasing family of finite subgroups $\{ G_k \}$ of $G$ such that $G=\bigcup_k G_k$.  We say that $G$ is \textit{locally finite}. Evidently any such group is non-compact (in the discrete topology) unimodular and amenable.
\begin{pro}\label{pro3.1}
Let $G$ be a locally finite group. Any probability $\mu$ on $G$ can be represented as a convex linear combination 
\begin{align}\label{equ31}
\mu=\sum_{k=0}^{\infty} c_k \mu_k,
\end{align}
of probabilities $\mu_k$ each of which is supported by a finite subgroup $G_k$. The sequence $\{G_k\}$ increases and $G=\bigcup_k G_k$.
\end{pro}
\begin{proof}
If $|supp \, \mu| < \infty$, then choose $c_0=1$, $c_k=0$ for $k>0$ and let $\{G_k\}$ be any increasing sequence of finite subgroups such that $supp \, \mu \subset G_0$.

Assume now that $| supp \, \mu| = \infty$. Choose any increasing sequence  $\{G_k\}$ of finite subgroups of $G$ such that $\mu(G_0)>0$. Then $\mu(G \backslash G_k)>0$ for all $k>0$ and let
\begin{align*}
\mu_k:=\left(\sum_{l=0}^k\frac{\mu (G_l \backslash G_{l-1})}{\mu (G \backslash G_{l-1})}\right)^{-1}\left( \sum_{l=0}^k \frac{\textbf{1}_{G_l \backslash G_{l-1}}}{\mu ( G\backslash G_{l-1})}\right) \mu
\end{align*}
and
\begin{align*}
c_k:=\mu (G_k \backslash G_{k-1}) \left(\sum_{l=0}^k\frac{\mu (G_l \backslash G_{l-1})}{\mu (G \backslash G_{l-1})}\right),
\end{align*}
where $G_{-1}:=\emptyset$. The proof is finished.
\end{proof}
\begin{rem}
It is easy to see that the representation (\ref{equ31})
is not unique. 
Fix any $\{c_k\}$ and $\{\mu_k\}$ such that (\ref{equ31}) holds and define a probability $Q$ on $\Omega=\mathbb{N}\times G$ as follows:
\begin{align*}
Q(\{k\}\times A) = c_k \mu_k(A).
\end{align*}
Evidently we have
\begin{align*}
Q(\mathbb{N} \times A)=\mu(A).
\end{align*}
Let $H_k:=\{k\}\times G$, then $Q(H_k)=c_k$, $H_k\cap H_l=\emptyset$, for  $k \neq l$, and $\Omega=\bigcup_k H_k$. Finally 
\begin{align*}
Q(\mathbb{N}\times A|H_k)=\mu_k(A) 
\end{align*}
and 
\begin{align*}
\mu(A)=Q(\mathbb{N}\times A)=\sum_{k=0}^{\infty} c_k Q(\mathbb{N}\times A|H_k)=\sum_{k=0}^{\infty} c_k\mu_k(A).
\end{align*}
Thus with probability $c_k$ we choose the group $G_k$ and then with probability $\mu_k$ we perform a choice of an element from this group. In the paper \cite{Bro} four interesting choices of $(\mu_k, c_k)$ (called shuffling models) are presented. Our choice is different from \cite{Bro}. We plan to come back to these shuffling models in our further publications.
\end{rem} 

In this work we consider the following special cases: 
Let $m_{k}$ be the normalized Haar measure on $G_{k}$. Let $c=\{c_{k}\}$ be a sequence of strictly positive reals such that $\sum_k c_{k}=1$. Define a probability measure $\mu$ on $G$ as follows
\begin{align*}
\mu=\sum_{k}c_{k}m_{k}. 
\end{align*}

Let $m$ be the counting measure on $G$ and $L^p=L^p(G,m)$. Define on $L^p$ the following linear operators $P_{k}f=f*m_{k}$ and $Pf=f*\mu$. Clearly $P_k$ and $P$ are bounded symmetric operators and
\begin{align*}
P=\sum_{k=0}^{\infty} c_{k}P_{k}.
\end{align*}
Let us compute $P^{n}$. Since for any $l,k \in \mathbb{N}$, $m_l*m_k=m_{\max(l,k)}$,
\begin{align*}
P^{n}=&\left(\sum_{k=0}^{\infty} c_{k}P_{k}\right)^{n}=\sum_{k_{1},\ldots,k_{n}=0}^{\infty}c_{k_{1}}\cdot\ldots\cdot c_{k_{n}}P_{k_{1}}\cdot\ldots\cdot P_{k_{n}}=
\end{align*}
\begin{align*}
=&\sum_{k_{1},\ldots,k_{n}=0}^{\infty}c_{k_{1}}\cdot\ldots\cdot c_{k_{n}}P_{\max(k_{1},\ldots,k_{n})}=\\
=&\sum_{k=0}^{\infty} \left(\sum_{k_{1},\ldots,k_{n}: \max(k_{1},\ldots,k_{n})=k}c_{k_{1}}\cdot\ldots\cdot c_{k_{n}}\right)P_{k}=\sum_{k=0}^{\infty}a_{k}P_{k},
\end{align*}
where, as in the proof of Proposition \ref{pro0},
\begin{align*}
a_{k}=\left( c_{0}+\ldots+c_{k} \right)^{n}-\left( c_{0}+\ldots+c_{k-1} \right)^{n}, \qquad \textrm{for}\quad k \geq 1
\end{align*}
and
\begin{align*}
a_{0}=c_{0}^{n}.
\end{align*}
From these equations we obtain
\begin{align*}
c_{k}=\left( a_{0}+\ldots+a_{k} \right)^{\frac{1}{n}}-\left( a_{0}+\ldots+a_{k-1} \right)^{\frac{1}{n}}, \qquad \textrm{for}\quad k \geq 1
\end{align*}
and
\begin{align*}
c_{0}=a_{0}^{\frac{1}{n}}.
\end{align*}
In particular, the symmetric operator $P^{\frac{1}{m}}$, defined via spectral theory, has the following representation:
\begin{align*}
P^{\frac{1}{m}}=\sum_{k=0}^{\infty} \left( \left( c_{0}+\ldots+c_{k} \right)^{\frac{1}{m}}-\left( c_{0}+\ldots+c_{k-1} \right)^{\frac{1}{m}}\right)P_{k}.
\end{align*}
As a consequence of these two observations we obtain:
\begin{align} \label{eq1}
P^{\frac{n}{m}}=\sum_{k=0}^{\infty} \left( \left( c_{0}+\ldots+c_{k} \right)^{\frac{n}{m}}-\left( c_{0}+\ldots+c_{k-1} \right)^{\frac{n}{m}}\right)P_{k}.
\end{align}
Let now $t\in\mathbb{R}_{+}$ and $r_{n}\in \mathbb Q_{+}$ be a sequence of rationals such that $r_{n}\rightarrow t$. Then, by spectral theory, $P^{r_{n}}\rightarrow P^{t}$ strongly,
and for all $k\geq 0$,
\begin{align*}
\left( c_{0}+\ldots+c_{k} \right)^{r_{n}}\rightarrow \left( c_{0}+\ldots+c_{k} \right)^{t}.
\end{align*}
Hence, passing to the limit in both sides of the equation (\ref{eq1}) we obtain
\begin{align}\label{equat3.3}
P^{t}=\sum_{k=0}^{\infty} \left( \left( c_{0}+\ldots+c_{k} \right)^{t}-\left( c_{0}+\ldots+c_{k-1} \right)^{t} \right) P_{k}.
\end{align}
The above facts are crucial for our purposes.
We summarize them in the following proposition.
\begin{pro}\label{proposit3.2}
The measure $\mu =\mu(c)$ is infinite divisible. In particular, there exists a weakly continuous convolution semigroup $(\mu_{t})_{t>0}$ of probability measures on $G$ such that $\mu=\mu_1$. Moreover the following representation holds:
\begin{align*}
\mu_{t}=\sum_{k=0}^{\infty} C_{k}(t)m_{k},
\end{align*}
where
\begin{align*}
C_{k}(t)=\left( c_{0}+\ldots+c_{k} \right)^{t}-\left( c_{0}+\ldots+c_{k-1} \right)^{t}, \qquad \textrm{for} \quad k\geq 1
\end{align*}
and
\begin{align*}
C_{0}(t)=c_{0}^{t}.
\end{align*}
\end{pro}
\begin{rem}
The problem of embedding  an infinite divisible probability measure $\mu$, defined on a general locally compact group, into a weakly continuous convolution semigroup $(\mu_t)_{t>0}$ is open. See the papers of S.G. Dani,  Y. Guivarc'h, and R. Shah \cite{Gui}, and McCrudden \cite{Mcc}. We provide a solution to this problem when  the underlying group $G$ is locally finite and the measure $\mu$ is a convex linear combination of idempotent measures.
\end{rem}
\begin{pro}\label{pro6}
For each $t>0$ and $x \in G$, we define $\mu_t(x):=\mu_t(\{x\})$. The function  $(t,x)\to \mu_{t}(x)$ has the following form 
\begin{align*}
\mu_{t}(x)=\sum_{k=0}^{\infty} \left(\sum_{n\geq k} \frac{C_{n}(t)}{|G_{n}|}\right)\mathbf{1}_{G_{k}\backslash G_{k-1}}(x),
\end{align*}
where we let $G_{-1}=\emptyset$. In particular,
\begin{align*}
\mu_{t}(e)=\sum_{n\geq 0}C_{n}(t)/|G_{n}|.
 \end{align*}
\end{pro}

Since $\mu=\mu_t|_{t=1}$, the celebrated L\'{e}vy-Khinchin formula \cite{Hey}, \cite[Thm. 18.19]{Bfo} applies: $\mu$ can be decomposed into Gaussian and Poissonian components. Since $G$ is discrete, the Gaussian component is degenerate. Thus $\mu_t$ and therefore $\mu$ are Poissonian measures. Again, since $G$ is discrete, the corresponding L\'{e}vy measure $\Pi$ is in fact a finite Borel measure on $G\backslash \{e\}$. We can extend this measure to the whole group by putting $\Pi(\{e\})=\pi_{0}>0$. This extension, by the L\'{e}vy-Khinchin formula, does not change the measure $\mu_{t}$. Thus, we must have the following representation
\begin{align*}
\mu_{t}=\exp(-t\pi)\left(\varepsilon_{0}+t\Pi+\frac{t^{2}}{2!}\Pi^{*2}+\ldots \right), \qquad
\pi=\Pi(G).
\end{align*}
\textit{We claim that the measure $\Pi$ has the same structure as $\mu$:}
\begin{align*}
\Pi=\sum_{k=0}^{\infty}\pi_{k}m_{k}.
\end{align*}
To prove this claim we have to find the coefficients $\{\pi_{k}\}_{k\geq 0}$. Note that $\pi_0$ has been already chosen. Define the operator $\Pi: L^2 \to L^2$ as
\begin{align*}
\Pi f:= f*\Pi , \quad f \in L^{2}.
\end{align*}
In this notation our claim takes the following form:
\begin{align*}
P^{t}=\exp(-t\pi)\cdot \exp(t\Pi), \quad \textrm{for} \quad t>0.
\end{align*}
We have
\begin{align*}
\exp(t\Pi)=\sum_{n=0}^{\infty}\frac{t^{n}}{n!}\Pi^{n}=\sum_{n=0}^{\infty}\frac{t^{n}}{n!}\left(\sum_{k=0}^{\infty}\pi_{k}(n)P_{k}\right),
\end{align*}
where, similarly to (\ref{equat3.3}), for $n\geq 1$ and $k \geq 1$,
\begin{align*}
\pi_k(n)=\left(\pi_0+\ldots+\pi_{k} \right)^{n}-\left(\pi_0+\ldots+\pi_{k-1} \right)^{n},
\end{align*}
and
\begin{align*}
\pi_0(0):=1,\qquad \pi_k(0)=0, \quad \textrm{for} \quad k\geq 1.
\end{align*}
Having this in mind we continue our computations:
\begin{align*}
\exp(t\Pi)=\sum_{k=0}^{\infty} \left( \sum_{n=0}^{\infty}\frac{t^{n}}{n!}\pi_{k}(n) \right)P_k.
\end{align*}
Let $k=0$, then
\begin{align*}
\sum_{n=0}^{\infty}\frac{t^{n}}{n!}\pi_{k}(n)=1+t\pi_0+\frac{t^{2}}{2!}\pi_{0}^{2}+\ldots=\exp(t\pi_0).
\end{align*}
For $k\geq 1$ we obtain
\begin{align*}
\sum_{n=0}^{\infty}\frac{t^{n}}{n!}\pi_{k}(n)=&\sum_{n=0}^{\infty}\frac{t^{n}}{n!} \left(\left(\pi_{0}+\ldots+\pi_{k}\right)^{n}-\left(\pi_{0}+\ldots+\pi_{k-1}\right)^{n} \right)=\\
=&\sum_{n=0}^{\infty}\frac{t^{n}}{n!} \left(\pi_{0}+\ldots+\pi_{k}\right)^{n}-\sum_{n=0}^{\infty}\frac{t^{n}}{n!} \left(\pi_{0}+\ldots+\pi_{k-1}\right)^{n}=\\
=&\exp(t\left(\pi_{0}+\ldots+\pi_{k}\right))-\exp(t\left(\pi_{0}+\ldots+\pi_{k-1}\right)).
\end{align*}
Putting together all the computations above we obtain
\begin{align*}
e^{-t\pi}\cdot e^{t\Pi}=e^{-t(\pi-\pi_0)}P_0+\sum_{k\geq 1}\left(e^{-t\left(\pi- (\pi_0+ \ldots +\pi_k)\right)} -e^{-t\left(\pi- (\pi_0+ \ldots +\pi_{k-1})\right)} \right)P_k.
\end{align*}
On the other hand,
\begin{align*}
e^{-t\pi}\cdot e^{t\Pi}=P^{t}=c_{0}^{t}P_{0}+\sum_{k\geq 1}\left((c_0+\ldots+c_{k})^{t}-(c_0+\ldots+c_{k-1})^{t}\right)P_k.
\end{align*}
This equality is true if and only if the sequence $\{\pi_k\}$ satisfies the following infinite system of equations:
\begin{align*}
\pi-(\pi_0+\ldots+\pi_k)=-\log (c_0+\ldots +c_k), \quad \textrm{for} \quad k=0,1,2,\ldots \,.
\end{align*}
This system uniquely defines $\pi_k$ for $k\geq 1$:
\begin{align*}
\pi_k=& \log (c_0+\ldots +c_k) -  \log (c_0+\ldots +c_{k-1})=\\
=&\log \left(1+\frac{c_k}{c_0+\ldots + c_{k-1}} \right)>0.
\end{align*}
Since $\pi_0>0$ has been chosen, we let $\pi=\pi_0+\pi_1+\pi_2+\ldots$ and define
\begin{align*}
\Pi=\sum_{k\geq 0}\pi_{k}P_{k}.
\end{align*}
Applying all the above computations in reverse order, we come to the following conclusion:
\begin{pro}[Poisson representation] \label{pro3.6} There exists a finite measure $\Pi=\sum_{k=0}^{\infty}\pi_k m_k$ such that
\begin{align*}
\mu_t= \exp(-t\pi)\cdot \left(\varepsilon_{0}+t\Pi+\frac{t^{2}}{2!}\Pi^{*2}+\ldots \right), \quad \pi=\Pi(G).
\end{align*}
\end{pro}

\section{Recurrence criterion.}\label{sec:rec}

The problem of recurrence of random walks on countable abelian
groups with finite number of generators has been investigated in
details in the book of Spitzer \cite{Spi}. Dudley \cite{Dud} and
Revuz \cite{Rev} proved that each countable abelian group which
does not contain a subgroup isomorphic $\mathbb{Z}^3$ admits a
recurrent random walk. The proof of Dudley is not constructive,
hence it is very desirable to find some explicit construction of
recurrent random walks on such groups. This construction has been
done in the works of S. Molchanov and his collaborators
\cite{Mol}, \cite{Kas} for the following groups: $G=\mathbb Z[1/p]$ - the group
of all rational numbers of the form
$r=\frac{k}{p^m}$ where $k$, $p>1$ and $m \geq 0$ are integers, and
$G=\Gamma_p=\mathbb{Z}^k\otimes \mathbb{Z}(p)^{(\infty)}$, with
$k=0,1,2$. They obtained sufficient conditions for the recurrence
and also necessary ones which are very close to each other.
Brofferio and Woess \cite{Bro} proved recurrence criteria for
certain card shuffling models, that is, random walks on the
infinite symmetric group $S_{\infty}=\bigcup_{n\geq 1}S_n$. One of
these models has been considered previously by Lawler \cite{Law}.
He obtained a very general sufficient condition for recurrence:
\begin{pro}\label{pro7}
A random walk on $G=\bigcup_{n\geq 0} G_n$ with law $\mu$ is recurrent, if
\begin{align*}
\sum_{n=1}^{\infty} \frac{1}{|G_n|(1-\mu(G_n))}=\infty.
\end{align*}
\end{pro}
Special cases of Proposition \ref{pro7} had been proved previously by
Flatto and Pitt \cite{Fla} (each $G_n$ is cyclic, or $G$
is the direct sum of finite abelian groups), and before that by
Darling and Erd$\ddot{\textrm{o}}$s \cite{Dar} ($G$ is the
direct sum of infinitely many copies of the group of order two). 

In the case $\mu=\mu(c)$, we
show that the condition above is also necessary.
\begin{pro}\label{pro3.5}
A random walk on $G$ with law $\mu=\mu(c)$
is recurrent if and only if the following condition holds:
\begin{align*}
\sum_{n=0}^{\infty}\frac{1}{|G_{n}|\left(1-\mu(G_{n})\right)}=\infty.
\end{align*}
\end{pro}
\begin{proof}
Let $\lambda \to N(\lambda)$ be a right-continuous, non-decreasing step-function having jumps at the points $\lambda_k=\sigma(k)$, where $\sigma(k)=\sum_{i>k}c_i$, and $N(\lambda_k)=1/|G_{k}|$. 
By Proposition \ref{pro6},
\begin{align*}
\mu^{*n}(e)=& \mu_n(e)=\sum_{k \geq 0}\frac{1}{|G_k|}\left[ (1-\sigma(k))^n-(1-\sigma(k-1))^n\right]=\\
=& \sum_{k \geq 0}\frac{1}{|G_k|} (1-\sigma(k))^n - \sum_{k \geq 0}\frac{1}{|G_k|} (1-\sigma(k-1))^n=\\
=& \sum_{k \geq 0}(1-\sigma(k))^n \left( \frac{1}{|G_k|} -\frac{1}{|G_{k+1}|}\right)=\int_0^{\infty} (1-\lambda)^n dN(\lambda).
\end{align*}
Since the measure $B \to \int_B dN$ is supported by the interval $[0, \sigma(0)]\subset [0,1]$, there exists $\delta >1$, such that, for all $n \in \mathbb{N}$
\begin{align}
\int_0^{\infty} e^{-n\delta\lambda} dN(\lambda)\leq \int_0^{\infty} (1-\lambda)^n dN(\lambda) \leq \int_0^{\infty} e^{-n\lambda} dN(\lambda).
\end{align}
Next we use a well-known criterion of recurrence: a random walk with law $\mu$ is recurrent if and only if
\begin{align*}
\sum_n \mu^{*n}(e)=\infty.
\end{align*}
Since 
\begin{align}\label{equat4.2}
\mu^{*n}(e)\stackrel{d}{\simeq} \int_0^{\infty} e^{-n\lambda} dN(\lambda),
\end{align}
we easily transform this criterion into the property
\begin{align*}
\sum_{k\geq 0} \frac{1}{|G_k|\sigma(k)}=\int_0^{\infty}\frac{dN(\lambda)}{\lambda}\asymp \sum_n \mu^{*n}(e)=\infty.
\end{align*}
Finally, observe that $1-\mu(G_n)\sim \sigma(n)$. This finishes the proof.
\end{proof}

Let $\{X(n)\}_{n\geq 0}$ be the random walk on $G=\bigcup_n G_n$ with law $\mu$. Assume that $\{X(n)\}$ is recurrent. Y. Guivarc'h asked the following question: \textit{How slow may  the function $n \to \mathbb{P}(X(2n)=e|X(0)=e)$ decay at infinity?}

It is well known that if a locally compact non-compact group $G$
is compactly generated (for example if $G$ is infinite and finitely generated), the upper  rate of decay, among all functions defined as 
$$n\to\parallel \mu^{*2n}\parallel_{\infty},$$ 
where $\mu$ is a symmetric probability density
whose support generates $G$, is realized (up to the equivalence relation $\stackrel{d}{\asymp} $) when $\mu$ has finite second moment with respect to a 
word metric defined by a compact symmetric generating set. This rate is a geometric invariant of the group $G$. See
for instance \cite{BPSa}, \cite{ChP}, \cite{ChP1}. In particular,
let $G$ be an abelian non-compact compactly generated group. By the structure
theory \cite[Thm.9.8]{Hew1},
\begin{align*}
G \cong \mathbb{R}^l \times \mathbb{Z}^m \times K, \qquad l+m>0,
\end{align*}
where $K$ is a compact group. Then, for any symmetric 
probability density $\mu:G\to \mathbb{R}_+$, whose support generates $G$, we must have
\begin{align*}
\parallel \mu^{*2n}\parallel_{\infty} \leq C n^{-(l+m)/2}
\qquad \textrm{at} \quad \infty.
\end{align*}
Our next proposition shows that if $G$ is \textit{not compactly generated}
the upper rate of decay of the function $n \to \parallel
\mu^{*2n}\parallel_{\infty}$ may not exist in the sense explained
above.

\begin{pro}\label{cor3} Let $G$ be an infinite countable locally finite group and $\mathcal{F}:
\mathbb{R}_{+} \to \mathbb{R}_{+}$ be a non-decreasing continuous 
function such that $\mathcal{F}(t)=o(t)$ and $\mathcal{F}(t) \to \infty$ at infinity.
There exists a symmetric strictly positive probability density
$\mu$ on $G$ such that 
\begin{align*}
-\log \mu^{*n}(e)/\mathcal{F}(n) \to 0 \quad \textrm{at} \quad \infty,
\end{align*}
(compare with Theorem \ref{cor1}).
\end{pro}
\begin{proof}
Let $G=\bigcup_{k=1}^{\infty}G_k$, where $\{e\}=G_0\subset G_1\subset \ldots \subset G_k\subset \ldots$ are finite subgroups of $G$. Let $\mu=\mu(c)$ for some $c=(c_k)$. Then by (\ref{equat4.2}), for some $\delta>1$
\begin{align*}
\mu^{*n}(e)\geq \int_0^{\infty} e^{-\delta n \lambda} dN(\lambda).
\end{align*}
Proceeding as in the proof of Theorem \ref{cor1} but choosing $N(\lambda)\geq e^{-\mathcal{M}(\delta \lambda)}$ with $\mathcal{M}$ continuous and decreasing, and such that
\begin{align*}
\mathcal{L}(\mathcal{M})/\mathcal{F}\to 0 \quad \textrm{ at }\quad \infty,
\end{align*}
we can find $c=(c_k)$ and hence $\mu=\mu(c)$ such that $c_k>0$ and
\begin{align*}
\mu^{*n}(e) \geq \int_0^{\infty} e^{-n\delta \lambda} dN(\lambda) \geq \int_0^{\infty} e^{-n \lambda} de^{-\mathcal{M}(\lambda)}.
\end{align*}
Therefore as $n\to \infty$ we obtain
\begin{align*}
-\log \mu^{*n}(e)/\mathcal{F}(n) \leq \mathcal{K}(\mathcal{M})/\mathcal{F}(n) \sim \mathcal{L}(\mathcal{M})/\mathcal{F}(n) \to 0.
\end{align*}
The proof of the proposition is finished.
\end{proof}
\begin{exa}
Assume that $t \to g(t)$ is a non-decreasing
function such that $\lim_{t\to 0} g(t) = 0$. Choose a sequence $c=(c_k)$ of strictly positive reals and a right-continuous step-function $N\geq g$ which has jumps at $\lambda_k=\sigma(k)$ and $N(\lambda_k)=1/|G_{k}|$. Let $\mu=\mu(c)$. Thanks to our
choice the following chain of inequalities holds true:
\begin{align*}
 \parallel \mu^{*n}\parallel_{\infty} & \geq 
\int_0^{\infty}e^{-n\delta \lambda}d{N}(\lambda) \geq
\int_0^{\infty}e^{-n\delta \lambda}dg(\lambda) =
\delta n \int_0^{\infty}e^{-n\delta \lambda}g(\lambda)
d\lambda=\\
 &=\delta\int_0^{\infty}e^{-\delta s}g\left(\frac{s}{n}\right)ds =
\delta g\left(\frac{1}{n}\right)\int_0^{\infty}\left[g\left(\frac{s}{n}\right)/g\left(\frac{1}{n}\right)\right]e^{-\delta s}ds
.
\end{align*}
Assuming that the ratio $g(\lambda \tau)/g(\tau)$ has dominated
convergence as $\tau \to 0$ to some integrable function (in fact, 
always to the function $\lambda \to \lambda^{\alpha}$) we obtain
\begin{align*}
\liminf_{n \to \infty} \frac{ \parallel
\mu^{*n}\parallel_{\infty}}{g(1/n)} \geq \delta
\int_0^{\infty}s^{\alpha} e^{-\delta
s}ds=\delta^{-\alpha}\Gamma(1+\alpha).
\end{align*}
For instance, choosing $g(t)\sim (\log_{(k)}\frac{1}{t} )^{-1}$ at $0$, we obtain for some $c>0$
\begin{align*}
\parallel \mu^{*n}\parallel_{\infty} \geq c \left( \log_{(k)} n \right)^{-1} \quad \textrm{at} \quad \infty.
\end{align*}
\end{exa}

\section{On the $L^{p}$-spectrum of the Laplacian.}

Let $X(n)$ be a random walk on $G$ with law $\mu=\mu(c)$ and $\left(P^n\right)$ its Markov semigroup.
The Laplacian $\Delta$ associated to $X(n)$ is defined as the Markov generator of the semigroup $\left(P^n\right)$, that is,  $\Delta=P-I$. Clearly $-\Delta:L^2\to L^2$ is a bounded, symmetric, non-negative definite operator.
Hence it admits the following representation
\begin{align*}
-\Delta=\int_0^{\infty} \lambda dE_{\lambda} \quad \textrm{(in the strong topology)},
\end{align*}
where $\lambda \to E_{\lambda}$ is the associated spectral
resolution. See, for instance, \cite{Lax}. Evidently the operator-valued function $\lambda \to E_{\lambda}$ can be expressed in terms of the sequence $P_{k}$: $f \to f*m_k$, $k=0,1,\ldots\,$.

Observe that, for each $k\geq 0$, $P_k: L^2 \to L^2$ is an orthoprojector.
Moreover, the sequence $P_k$ decreases and $P_k \to 0$
strongly. Indeed, let $\delta_a$ be the function which takes value $1$ at $x=a$ and $0$ otherwise.
For $f=\delta_a$, we have
\begin{align*}
\parallel P_kf \parallel^2=(P_kf,f)=1/|G_k| \to 0 \quad \textrm{ as } \quad k\to
\infty.
\end{align*}
Hence the same is true for $f\in A$, the set of all finite linear
combinations of the elements $\{\delta_a\}_{a \in G}$. Since $A$ is
dense in $L^2$ and $\parallel P_k \parallel=1$ for all $k$, the
result follows.

Put $\sigma(k)=\sum_{i>k}c_i$ and define an operator-valued function $\lambda \to E_{\lambda}$
as follows
\begin{align*}
E_{\lambda}=\left\{ \begin{array}{ll} 0, & \quad -\infty
< \lambda \leq 0,
\\ P_k & \quad \sigma(k) \leq \lambda < \sigma(k-1),\quad \textrm{for} \quad k\geq 1,\\
I & \quad \sigma(0) \leq \lambda < +\infty.\\
\end{array} \right.
\end{align*}

It is easy to see that $\lambda \to E_{\lambda}$ is a spectral resolution, that is:
\begin{itemize}
\item $E_{\lambda}$ is an orthoprojector, for each $\lambda$, 
\item $E_{\lambda_1}E_{\lambda_2}=E_{\min\{\lambda_1,\lambda_2\}}$, for any $\lambda_1$, $\lambda_2$,
\item $E_{\lambda}\to 0$ at $-\infty$ and $E_{\lambda}\to
I$ at $+\infty$ (strongly),
\item $\lambda \to E_{\lambda}$ is right-continuous.
\end{itemize}
The operator-valued function $\lambda \to E_{\lambda}$ has jumps at the points
$\lambda_k =\sigma (k)$, $k=0,1,2,\dots$ and only at these points.
The heights of the jumps of the function $\lambda \to E_{\lambda}$ at the points $\sigma(k)$ are
\begin{align}\label{equat5.1}
E_{\sigma(k)}-E_{\sigma(k)-}=E_{\sigma(k)}-E_{\sigma(k+1)}=P_k-P_{k+1}.
\end{align}
\begin{pro}\label{pro4.1}
Let  $\lambda \to E_{\lambda}$ be as above, then
\begin{align*}
-\Delta=\int_0^{\infty} \lambda dE_{\lambda}.
\end{align*}
\end{pro}
\begin{proof}
The following chain of equalities holds true
\begin{align*}
-\Delta=& \sum_{k=0}^{\infty} c_k (I-P_k)=\sum_{k=0}^{\infty}
[\sigma(k-1)-\sigma(k)](E_{\sigma(0)}-E_{\sigma(k)})=\\=&\sum_{k=0}^{\infty}\left(E_{\sigma(0)}-E_{\sigma(k)} \right)\int_{\sigma(k)}^{\sigma(k-1)}d\lambda=\int_0^{\infty}
(E_{\sigma(0)}-E_{\lambda})d\lambda=\int_0^{\infty} \lambda
dE_{\lambda},
\end{align*}
where in the last equality we use the integration by parts
formula \cite{Shi}. 
\end{proof}

It is well-known that if $A$ is a Markov generator and $0<\alpha \leq 1$ then $-(-A)^{\alpha}$ (more generally, $-\phi(-A)$, where $\phi$ is a Bernstein function\footnote{In the discrete-time setting we assume that $\phi(0)=0$ and $\phi(1)=1$.}) is again a Markov generator.  See e.g. \cite{BenSC}, \cite{L}, \cite{Shi}. If $\alpha>1$, then $-(-A)^{\alpha}$ is not a Markov generator in general. For instance, let $P$ be the transition operator of the simple random walk on $\mathbb{Z}$ and let $A=P-I$ be its Markov generator. Define $\bar{A}:=-(-A)^2$. $\bar{A}$ is not a Markov generator because its transition operator $\bar{P}=I+\bar{A}$ is not positivity-preserving:
\begin{align*}
\bar{P}\textbf{1}_{\{0\}}(0)=(2P-P^2)\textbf{1}_{\{0\}}(0)=-P^2\textbf{1}_{\{0\}}(0)=-\frac{1}{2}<0.
\end{align*}
Similarly, using Taylor expansion, one can show that $-(-A)^{\alpha}$ is not a Markov generator for any $\alpha>1$ in this example.

\begin{pro} Let $X(n)$ be the random walk on $G$ with law $\mu=\mu(c)$. Let $P$ be its transition operator and let $\Delta=P-I$ be its Laplacian. For any continuous increasing function $\phi:[0,1]\to [0,1]$ such that $\phi(0)=0$ and $\phi(1)=1$, the operator $-\phi(-\Delta)$ is a Laplacian, that is,
\begin{align*}
-\phi(-\Delta)=P_{\phi}-I,
\end{align*}
where
\begin{align*}
P_{\phi}f=f*\phi(\mu)
\end{align*}
and
\begin{align*}
 \phi(\mu)= \sum_{k \geq 0}\left\{\phi(\sigma(k-1))-\phi(\sigma(k)) \right\} m_k.
\end{align*}
In particular, for any $\alpha>0$, $-(-\Delta)^{\alpha}$ is the Laplacian defined by the random walk with law $\mu(c(\alpha))$,
\begin{align*}
\mu (c(\alpha))=\sum_{k \geq 0} \left\{(\sigma(k-1))^{\alpha}-(\sigma(k))^{\alpha}\right\} m_k.
\end{align*}
\end{pro}
\begin{proof}
Proposition \ref{pro4.1} and the integration by parts formula show that
\begin{align*}
P_{\phi}:=&I-\phi(-\Delta)=\int_0^1(1-\phi(\lambda))dE_{\lambda}=\int_0^1 E_{\lambda}d\phi(\lambda)=\\
=&\sum_{k \geq 0}\left\{\phi(\sigma(k-1))-\phi(\sigma(k))\right\}E_{\sigma(k)}=\\
=&\sum_{k \geq 0}\left\{\phi(\sigma(k-1))-\phi(\sigma(k))\right\}P_k.
\end{align*}
It follows that
\begin{align*}
P_{\phi}(f)=&\sum_{k \geq 0}\left\{\phi(\sigma(k-1))-\phi(\sigma(k))\right\}P_k(f)=\\
=&\sum_{k \geq 0}\left\{\phi(\sigma(k-1))-\phi(\sigma(k))\right\}f*m_k=\\
=& f*\left(\sum_{k \geq 0}\left\{\phi(\sigma(k-1))-\phi(\sigma(k))\right\}m_k\right)=f*\phi(\mu).
\end{align*}
The proof is finished.
\end{proof}

\begin{pro}
The spectrum $Spec_{L^2}(-\Delta)$ of the self-adjoint operator $-\Delta: L^2 \to
L^2$ is of the form:
\begin{align*}
Spec_{L^2}(-\Delta)=\{0, \ldots, \sigma(k), \ldots, \sigma(1),
\sigma(0)\}.
\end{align*}
Each $\sigma(k)$ is an eigenvalue of $-\Delta$ having infinite
multiplicity. 
\end{pro}
\begin{proof}
That each $\sigma(k)$ is an eigenvalue of the operator $-\Delta$ follows from Proposition \ref{pro4.1}. According to equation (\ref{equat5.1}), the corresponding space $\mathcal{H}_k$ of $\sigma(k)$-eigenfunctions is  $\mathcal{H}_k=(P_{k}-P_{k+1})L^2$. What is left is to show that 
$\mathcal{H}_k$ is
infinite dimensional. For
$a \in G$ and $k=0,1,\ldots$ define $f_{k,a}=(P_{k}-P_{k+1})\delta_a$, and write
\begin{align*}
P_l\delta_a(x)=& \delta_a*m_l(x)=\int_{G_l}
\delta_a(xy)dm_l(y)
=\int_{G_l}\delta_e(a^{-1}xy)dm_l(y)= \\ =& \left\{
\begin{array}{cl} 1/|G_l|, & \quad \textrm{if } x\in a \cdot G_l
\\ 0 & \quad \textrm{otherwise}\\
\end{array} \right. = \frac{1}{|G_l|}\textbf{1}_{a\cdot G_l}(x).
\end{align*}
Hence, the function $f_{k,a}$ is supported by the set $a \cdot G_{k+1}$. In
particular, for each $a,b \in G$,
\begin{align*}
supp f_{k,a} \cap supp f_{k,b}=\emptyset
\end{align*}
if and only if $(aG_{k+1}) \bigcap (bG_{k+1})=\emptyset$, equivalently
$ab^{-1} \notin G_{k+1}$.
Define a distance $\rho(x,y)$ on $G$
as follows
\begin{align*}
\rho(x,y)=n \quad \textrm{ if and only if } \quad xy^{-1}\in G_n \backslash
G_{n-1}.
\end{align*}
Evidently $\rho$ is a distance on $G$, and $(G, \rho)$ is a complete metric space. For any $n\in \mathbb{N}$, the subgroup
$G_n$ is a ball of radius $n$ centered at $e$. Now
choose a sequence $\{a_l\}\subset G$ such that
$\rho(e,a_l)=(k+2)l$. Then for any $l>m$,
\begin{align*}
\rho (a_l, a_m)\geq (k+2)l-(k+2)m=(k+2)(l-m)>(k+1).
\end{align*}
Hence for any $l\neq m$, the functions $f_{k, a_l}$ and $f_{k,a_m}$
have disjoint supports and therefore are orthogonal in
$\mathcal{H}_k\subset L^2$. Thus $\mathcal{H}_k$ contains an infinite sequence of
orthogonal elements, hence it is infinite dimensional.
\end{proof}
\begin{pro}\label{pro4.3}
$-\Delta: L^p \to L^p$ is a bounded operator and\\
$1$. $Spec_{L^p}(-\Delta)=Spec_{L^2}(-\Delta)$, for any $p \geq 1$.\\
$2$. (Strong Liouville property) Any function $u\geq 0$ such that $\Delta u=~0$ must be a constant-function. In particular, $0 \in Spec_{L^{\infty}}(-\Delta)$ is an eigenvalue of multiplicity one.
\end{pro}
\begin{proof}[Proof of 1:]
1. For $p\geq 1$ and $f\in L^p$, $-\Delta f=f-f*\mu \in L^p$. Hence $-\Delta:L^p\to L^p$ is a bounded operator with norm $\parallel -\Delta \parallel_{L^p\to L^p}\leq 2$.
 
Let $f \in \bigcap_{p=1}^{\infty}L^p$ and $\lambda \notin Spec_{L^2}
(-\Delta)$, then
\begin{align*}
-\Delta f = \sum_{k=0}^{\infty} \sigma(k)(P_k-P_{k+1})f=\sum_{k=0}^{\infty}\sigma(k)(f*m_k-f*m_{k+1}).
\end{align*}
It follows that
\begin{align*}
& \left(-\Delta  - \lambda I \right)^{-1}f = \sum_{k=0}^{\infty}
\frac{1}{\sigma(k)-\lambda} (P_k -P_{k+1})f=\\
& = \sum_{k=0}^{\infty}\frac{1}{\sigma(k)-\lambda} (f*m_k
-f*m_{k+1})=\\
& =  \sum_{k=0}^{\infty}\frac{f*m_k}{\sigma(k)-\lambda} -
\sum_{k=0}^{\infty}\frac{f*m_{k+1}}{\sigma(k)-\lambda}= \\
&= A(f) + \sum_{k\geq k_0}\frac{f*m_k}{\sigma(k)-\lambda} -
\sum_{k\geq k_0 -1}\frac{f*m_{k+1}}{\sigma(k)-\lambda}=\\
&= A(f)+ \sum_{k \geq k_0} \left[\frac{1}{\sigma(k)-\lambda} -
\frac{1}{\sigma(k-1)-\lambda} \right]f*m_k=\\
&=A(f)+ \sum_{k \geq
k_0}\frac{c_k}{(\sigma(k)-\lambda)(\sigma(k-1)-\lambda)}f*m_k=A(f)+B(f),
\end{align*}
where $k_0>1$ is chosen such that $\sigma(k)< \lambda$ for $k \geq k_0$. According to our choices, 
$A(f)=f*\underline{m}_{\lambda}$, where $\underline{m}_{\lambda}$ is a finite sum of signed measures having
finite variance. Hence, $A$ is a bounded operator in all $L^p$,
$1\leq p \leq \infty$. Next write
\begin{align*}
B(f)=& \sum_{k \geq k_0}
\frac{c_k}{[\sigma(k)-\lambda ][\sigma(k-1)-\lambda]}f*m_k=\\
=& f*\left( \sum_{k \geq k_0}
\frac{c_k}{[\sigma(k)-\lambda ][\sigma(k-1)-\lambda]}m_k\right)= f*\overline{m}_{\lambda},
\end{align*}
where the measure $\overline{m}_{\lambda} \asymp \sum_{k \geq k_0} c_km_k$ has finite
variance. Hence $B:L^p\to L^p$, $1\leq p \leq \infty$ is a bounded
operator. Thus, for $\lambda \notin Spec_{L^2}(-\Delta)$,
\begin{align*}
& \left(-\Delta  - \lambda I \right)^{-1}:\, \bigcap_{p \geq 1} L^p \to  \bigcap_{p \geq 1} L^p
\end{align*}
is uniformly bounded. Hence $\lambda \notin Spec_{L^p}(-\Delta)$ for all $1\leq p \leq \infty$.
On the other hand all functions $f_{k,a}=(P_k-P_{k+1})\delta_a$ are eigenfunctions of $-\Delta: L^p \to L^p$ corresponding to the eigenvalues $\sigma(k)$, $k=0,1,\dots \,$.\\
\textit{Proof of 2}: For any $k=0,1,\ldots\,$, define $u_k:=(P_k - P_{k+1})u$. We have $\Delta u_k=-\sigma(k)u_k$. On the other hand $\Delta u_k=\Delta (P_k - P_{k+1})u=(P_k - P_{k+1})\Delta u=0$. Hence $u_k\equiv 0$, for any $k=0,1,\ldots\,$. It follows that \begin{align*}
P_ku=P_{k+1}u, \quad k=0,1,\ldots\,.
\end{align*}
In particular,
\begin{align*}
u=P_ku=u*m_k, \quad k=0,1,\ldots\,.
\end{align*}
Observe that for any $a,x \in  G_k$, $P_ku(xa)=P_ku(x)$. In particular, $P_ku(x)=P_ku(e)$, for any $x \in G_k$. Thus, $u$ must be constant on each set $G_k$. It follows that $u$ is constant on $G=\bigcup_k G_k$.
\end{proof}
\begin{rem}
The second statement in Proposition \ref{pro4.3} implies that the set of bounded $\Delta$-harmonic functions (the Poisson boundary) is trivial. In the paper \cite{Kai}, V. Kaimanovich constructed an example of a probability measure $\mu$ on the infinite symmetric group $S_{\infty}$ having a non-trivial Poisson boundary. Of course the measure $\mu$ can not be represented as $\mu=\mu(c)$, for any $c=(c_k)$.
\end{rem}

\section{$L^{2}$-isospectral profile.}

Let $P$ and $\Delta=P-I$ be the transition operator and the
Laplacian defined by the probability measure $\mu=\mu(c)$.
For any subset $U\subset G$ we define the truncated operators
$P_U$ and $\Delta_U$ as follows:
\begin{align*}
P_Uf(x)=\textbf{1}_U(x)P(\textbf{1}_Uf)(x)
\end{align*}
and
\begin{align*}
\Delta_Uf(x)=\textbf{1}_U(x)\Delta(\textbf{1}_Uf)(x).
\end{align*}
Let $m_U$ be the counting measure on $U$ and $L^2(U)=L^2(U,m_U)$. Evidently, the
operators $P_U$ and $-\Delta_U$ can be regarded as operators on
$L^2(U)$, each of which is a bounded and symmetric, and $-\Delta_U$ is non-negative
definite. Moreover, $\Delta_U=P_U-I$.
If $U$ is a finite set, $L^2(U)$ is finite dimensional and its dimension $n$ is equal to $|U|$. 
Therefore the spectrum $Spec_{L^2} (-\Delta_U)$ of
$-\Delta_U$ consists of a finite number of points
$0<\lambda_1(U)\leq \ldots \leq \lambda_n(U)<1$,
repeated according to their multiplicity.

\textit{We define the $L^2$-isospectral profile $\lambda \to
\Lambda(\lambda)$ as follows:}
\begin{align}
\Lambda(\lambda):= \inf_{U:\, |U|\leq \lambda}
\lambda_1(U).
\end{align}

For any function $0<\Lambda_*\leq \Lambda$ the following Faber-Krahn type inequality holds true:
\begin{align*}
\lambda_1(U) \geq \Lambda_*(|U|), \quad \textrm{ for any finite } U \subset G.
\end{align*}
In the general setting of Markov generators this inequality was introduced in \cite{A} (see also \cite{GriTheHeat}, \cite{GriHeatKernel}) to investigate various aspects of the heat kernel behavior. In particular, it was proved in \cite{GriHeatKernel} that under some regularity assumptions on $\Lambda_*$, the Faber-Krahn type inequality is equivalent to the heat kernel estimate 
\begin{align*}
\sup_{x,y\in G} h(t;x,y)\leq \frac{1}{\Phi(t)}, \quad \forall t>0,
\end{align*}
where the functions $\Lambda_*$ and $\Phi$ are related by
\begin{align*}
t=\int_0^{\Phi(t)}\frac{d\lambda}{\lambda \Lambda_*(\lambda)}.
\end{align*}
See \cite{D}, \cite{E}, \cite{F}, \cite{ChP} and \cite{H}.

The main aim of this section is to obtain a lower bound for the function $\lambda \to
\Lambda(\lambda)$  in terms of the
sequences $\{c_k\}$ and $\{|G_k|\}$. Under certain regularity assumptions on $\{c_k\}$ we will obtain two-sided bounds for $\Lambda$ which are comparable in the sense of dilatational equivalence. See Theorem \ref{thm5.1} and Theorem \ref{thm2} below.

We define the function $\mathcal{T}:
\mathbb{R}_{+} \to \mathbb{R}_{+}$ as follows:
\begin{align}
\mathcal{T}(u):= 1- \sum_{i=0}^{\infty} c_i \left( 1 \wedge \frac{u}{|G_i|} \right).
\end{align}

The next proposition easily follows from the very definition of the function $u \to \mathcal{T}(u)$.
\begin{pro}\label{pro2}
The following properties hold true.\\
1) $\mathcal{T}$ is a continuous function.\\
2) $\mathcal{T}$ strictly decreases to zero at infinity.\\
3) $\mathcal{T}$ is convex.\\
4) For $\,|G_k| \leq u < |G_{k+1}|$,
\begin{align*}
\frac{1}{2} \sigma(k+1)< \mathcal{T}(u) < \sigma(k).
\end{align*}
\end{pro}

\begin{thm}\label{thm5.1}
For any finite set $U\subset G$, the following inequality holds:
\begin{align*}
\lambda_1(U) \geq \mathcal{T}(|U|).
\end{align*}
In particular, for any $\lambda \geq 1$,
\begin{align*}
\Lambda(\lambda) \geq \mathcal{T}(\lambda).
\end{align*}
\end{thm}
\begin{proof}
Choose a finite set $U\subset G$, a function $f$ such that $supp f\subset U$, and write
\begin{align*}
(f, P_if)=&\sum_{x\in U} f(x)P_if(x)=\sum_{x,y\in U} f(x)f(y) m_i(y^{-1}x)=\\
=& \frac{1}{|G_i|} \sum_{x,y \in U: \, y^{-1}x \in G_i} f(x)f(y).
\end{align*}
Let $[G:G_i]$ be the set of all co-sets $A=aG_i$, $a \in G$. Then,
\begin{align*}
 & \sum_{x,y \in U: \, y^{-1}x \in G_i} f(x)f(y)=\sum_{A\in [G:G_i]} \sum_{x,y \in A \cap
 U}f(x)f(y)=\\
& =\sum_{A\in [G:G_i]} \left( \sum_{x \in A \cap U} f(x)\right)^2
\leq  \sum_{A\in [G:G_i]} |A \cap U| \sum_{x \in A \cap U} f(x)^2 \leq \\
& \leq  \max_{A \in [G:G_i]} |A \cap U| \sum_{A \in [G:G_i]}
\sum_{x\in A \cap U} f(x)^2= \max_{A \in [G:G_i]} |A \cap U| \parallel f\parallel _{L^2(U)}^2 \leq \\
& \leq ( |G_i| \wedge |U|) \parallel f\parallel _{L^2(U)}^2.
\end{align*}
Hence, for any $i=0,1,2,\ldots \,$, and any finite set $U\subset
G$ the operator norm of the truncated operator $P_{i,U}$ can be
estimated as follows:
\begin{align*}
\parallel  P_{i,U} \parallel =& \sup \{ (f,P_if): \, supp f \subseteq U, \parallel f\parallel _{L^2}=1\}\leq \\
\leq & \frac{1}{|G_i|} ( |G_i| \wedge |U|)= \left( 1 \wedge  \frac{|U|}{|G_i|}\right).
\end{align*}
It follows, that
\begin{align*}
\parallel  P_U \parallel \leq \sum_{i=0}^{\infty} c_i \parallel  P_{i,U} \parallel \leq \sum_{i=0}^{\infty} c_i \left( 1 \wedge  \frac{|U|}{|G_i|}\right).
\end{align*}
With this inequality in hands the computations of $\lambda_1(U)$ become straightforward:
\begin{align*}
\lambda_1 (U)=& \min \{ (-\Delta_Uf,f): \, supp f \subseteq U, \parallel f\parallel_{L^2}=1\}=\\
=& \min \{ 1-(P_Uf,f): \, supp f \subseteq U, \parallel f\parallel_{L^2}=1\}=\\
=& 1- \max \{ (P_Uf,f): \, supp f \subseteq U, \parallel f\parallel_{L^2}=1\}=\\
=& 1- \parallel P_U \parallel \geq 1- \sum_{i=0}^{\infty} c_i \left( 1 \wedge  \frac{|U|}{|G_i|}\right) = \mathcal{T}(|U|).
\end{align*}
Since $\lambda \to \mathcal{T}(\lambda)$ is a decreasing function, we obtain
\begin{align*}
\Lambda(\lambda)=\inf \{ \lambda_1 (U): \, |U|\leq
\lambda \} \geq \mathcal{T}(\lambda).
\end{align*}
The proof is finished.
\end{proof}
\begin{pro}\label{proposit6.3}
For any $k=0,1,2,\ldots \,$,
\begin{align*}
\lambda_1(G_k)=\mathcal{T}(|G_k|).
\end{align*}
In particular, $\frac{1}{2} \sigma(k) < \lambda_1(G_k) < \sigma(k)$.
\end{pro}
\begin{proof}
Since $G_k$ is a subgroup of $G$, one can regard $P_{G_k}$ as a convolution operator on $G_k$. Indeed, for $x \in G_k$ and $f$ such that $supp f \subseteq G_k$ we have
\begin{align*}
P_{G_k}f(x)=&\int_{G_k} f(xy)d\mu(y)= \int_{G_k} f(xy)d\mu_{G_k}(y)=\\
=&f(x)*\mu_{G_k},
\end{align*}
where $\mu_{G_k}$ is the restriction of the probability measure $\mu$ on $G_k$. We have:
\begin{align*}
\mu_{G_k}=&\sum_{i=0}^{\infty} c_i m_{i, G_k} = \sum_{i=0}^{k-1} c_i m_i + \sum_{i=k}^{\infty} c_i \frac{|G_k|}{|G_i|}m_k=\\
=& \sum_{i=0}^{k-1}c_im_i + \left(\sum_{i=k}^{\infty}c_i \frac{|G_k|}{|G_i|} \right)m_k.
\end{align*}
In particular, since $|G_k|/|G_i|\leq 1/2$, for $i>k$,
\begin{align*}
\mu_{G_k}(1)= \sum_{i=0}^{k-1} c_i + \left(\sum_{i=k}^{\infty}c_i
\frac{|G_k|}{|G_i|} \right) \leq  \sum_{i=0}^{k}c_i +\frac{1}{2}
\sum_{i=k+1}^{\infty} c_i <1.
\end{align*}
Since for any finite group (more generally, any amenable group) the
norm of the convolution operator equals the variation of the
corresponding measure, we have
\begin{align*}
\parallel P_{G_k}\parallel =\mu_{G_k}(1).
\end{align*}
It follows that
\begin{align*}
\lambda_1(G_k) = 1- \parallel P_{G_k} \parallel =1-\sum_{i \leq k
}c_i - \sum_{i>k} c_i \frac{|G_k|}{|G_i|}=\mathcal{T} (|G_k|).
\end{align*}
We also have
\begin{align*}
\lambda_1(G_k)=\sum_{i>k} \left( 1 - \frac{|G_k|}{|G_i|} \right) c_i.
\end{align*}
Since $|G_k|/ |G_i|\leq 1/2$ for all $i>k$, we obtain
\begin{align*}
\frac{1}{2} \sum_{i>k} c_i < \lambda_1 (G_k)< \sum_{i>k} c_k.
\end{align*}
The proof is finished.
\end{proof}

According to Theorem \ref{thm5.1} the function $u \to \mathcal{T}(u)$ is a lower bound for the
function $u \to~\Lambda(u)$. Following ideas of F{\o}lner (see e.g. \cite{F}, \cite{Ers}, \cite{PitSal})
we will give an upper bound for the function $u\to
\Lambda(u)$. Define
\begin{align}\label{6.3}
k(n):=\min \{ k: \, \lambda_1(G_k)\leq 1/n^2\}
\end{align}
and let $v\to F(v)$ be the continuous piecewise linear function such
that $F(n)=|G_{k(n)}|$. Observe that $v\to F(v)$ is a strictly
increasing continuous function. Hence the inverse $v \to
F^{-1}(v)$ exists in the usual sense. We define the function $\Lambda_F: (1, +\infty)\to \mathbb{R}_+^{1}$ as follows:
 \begin{align}\label{6.4}
\Lambda_F(v):=\left( F^{-1}(v)-1 \right)^{-2}.
\end{align}
\begin{pro}\label{proposit6.4}
The following inequality holds true:
 \begin{align*}
\Lambda(v) \leq \Lambda_F(v), \qquad \textrm{for all} \quad v>1.
\end{align*}
\end{pro}
\begin{proof}
Let $\Omega_n:=G_{k(n)}$ and $\omega_v:=\Omega_{[F^{-1}(v)]}$. We have
\begin{align*}
|\omega_v|=|\Omega_{[F^{-1}(v)]}|=F\left([F^{-1}(v)]\right) \leq v.
\end{align*}
Using the definition of $k(n)$ we obtain:
\begin{align*}
\lambda_1(\omega_v)=\lambda_1 \left( \Omega_{[F^{-1}(v)]}\right)\leq (F^{-1}(v)-1)^{-2}=\Lambda_F(v).
\end{align*}
We conclude that
\begin{align*}
\Lambda(v)= \inf \{\lambda_1(U): \, |U|\leq v \}\leq
\lambda_1(\omega_v)=\Lambda_F(v).
\end{align*}
The proof is finished.
\end{proof}

The following regularity condition will play a crucial role in our further considerations.

\textbf{(A)} There exists $\lambda>0$ such that
\begin{align*}
c_k \leq \lambda \sigma(k), \quad \textrm{for all } \quad k \in \mathbb{N},
\end{align*}
equivalently,
\begin{align*}
\sigma(k-1) \leq (1+\lambda) \sigma(k), \quad \textrm{for all } \quad k \in \mathbb{N}.
\end{align*}

\begin{pro}\label{pro4}
Assume that Condition (A) holds. Let $\sigma: \mathbb{R}_{+} \to \mathbb{R}_{+} $ be any continuous decreasing extension of the function $\sigma: \mathbb{Z}_{+} \to \mathbb{R}_{+} $. Then the function $x\to
\sigma \circ \log x$ is doubling, that is, there
exists a constant $c>0$ such that for all $ x>1$,
\begin{align*}
c\left(\sigma\circ \log \right)(x) \leq \left(\sigma \circ \log \right)(2x) < 
\left(\sigma\circ \log\right)(x).
\end{align*}
\end{pro}
\begin{proof}
We have
\begin{align*}
\sigma(k+1)=\sigma(k)-c_{k+1} \geq \sigma(k)- \lambda \sigma(k+1),
\end{align*}
hence
\begin{align*}
\sigma(k+1) \geq \frac{1}{1+\lambda}\sigma(k).
\end{align*}
It follows that for  $k \leq \log x < k+1$,
\begin{align*}
\left(\sigma \circ  \log\right)(2x)=& \sigma\left(\log x +\log 2\right)\geq
\sigma(\log x +2) > \sigma(k+3)  \geq\\
\geq &  \frac{1}{(1+\lambda)^3}\sigma(k) \geq
\frac{1}{(1+\lambda)^3} \left(\sigma\circ \log \right)(x).
\end{align*}
The proof is finished.
\end{proof}
\begin{exa}\label{ex0} We present here  few examples illustrating Condition (A) introduced above.\\
1) Let $c_k\asymp q^{k^p}$, $0<q<1$, at $\infty$. \textit{Let us show that Condition (A) holds true if
and only if $0<p\leq 1$.} Indeed, assuming for simplicity of
notation that $q=e^{-1}$, we will have 
\begin{align*}
 c_{k} \asymp  \exp(-(k)^p) \qquad  \textrm{at} \quad \infty.
\end{align*}
Let $p\leq 1$, then by our assumption, for some $b_1>0$,
\begin{align*}
c_k/\sigma(k) \leq c_k/c_{k+1} \leq b_1 \exp((k+1)^p-k^p).
\end{align*}
Since
\begin{align*}
(k+1)^p-k^p=p \int_k^{k+1} \tau^{p-1} d\tau\leq p,
\end{align*}
we conclude that $c_k/\sigma(k) \leq \lambda$ for some $\lambda>0$ and  for  all $ k\in \mathbb{N}$.\\
Let $p>1$, then
\begin{align*}
\sigma(k)=c_{k+1} \left(1+\sum_{l\geq 1}c_{k+1+l}/c_{k+1}\right).
\end{align*}
By our assumption, for some $b_2 >0$,
\begin{align*}
c_{k+1+l}/c_{k+1}\leq b_2 \exp\{-[(k+1+l)^p-(k+1)^p]\}.
\end{align*}
Since $p>1$, we obtain
\begin{align*}
(k+1+l)^p-(k+1)^p=p\int_{k+1}^{k+1+l} \tau^{p-1} d\tau > pl.
\end{align*}
It follows that $\sigma(k) \asymp c_{k+1}$. Hence  $c_k/\sigma(k) \asymp c_k/c_{k+1} \to \infty$ at $\infty$.\\
2) Let $c_k \asymp k^{-p}$, $p>1$. Then $\sigma(k)\asymp k^{-p+1}$, therefore  $c_k/\sigma(k) \to 0$ at~$\infty$. Thus, Condition (A) is satisfied.\\
3) Let $c_k \asymp 1/(k\cdot\log k \cdot \log_{(2)}k \cdot \ldots \cdot (\log_{(n)}k)^p)$, $p>1$. Then 
\begin{align*}
\sigma(k) \asymp (\log_{(n)} k)^{-p+1}, 
\end{align*}
therefore $c_k/\sigma(k) \to 0$ at $\infty$. Condition (A) is satisfied.
\end{exa}
\begin{pro}\label{pro1}
Assume that Condition (A) holds. Then, there exists $c>0$ such that
\begin{align*}
\Lambda_F(v) \leq  c\mathcal{T}(v), \qquad \textrm{for any} \quad v > 1.
\end{align*}
\end{pro}
\begin{proof}
For any $n\geq 1$ and $n\leq v \leq n+1$, we have
\begin{align*}
\mathcal{T}(F(v)) \geq
\mathcal{T}(F(n+1))=\mathcal{T}(|{G}_{k(n+1)}|)=
\lambda_1({G}_{k(n+1)}) \geq \frac{1}{2} \sigma(k(n+1)).
\end{align*}
Thanks to our assumption, for some $c_1>0$,
\begin{align*}
\sigma(k(n+1)) \geq c_1 \sigma(k(n+1)-1),
\end{align*}
hence for some $c_2>0$,
\begin{align*}
 \mathcal{T}(F(v)) & \geq \frac{c_1}{2} \sigma\left(k(n+1)-1\right) \geq
\frac{c_1}{2}\lambda_1\left({G}_{k(n+1)-1}\right) \geq\\
&\geq \frac{c_1}{2} \frac{1}{(n+1)^2} \geq \frac{c_1}{4}
\frac{1}{v^2}=\frac{c_2}{v^2}.
\end{align*}
Since $u\to F(u)$ is strictly increasing and since $F(\infty)=\infty$, there exists $c_3>0$
such that,
\begin{align*}
\mathcal{T}(u) \geq
\frac{c_2}{\left(F^{-1}(u)\right)^2}\geq
\frac{c_3}{\left(F^{-1}(u)-1\right)^2}=c_3 \Lambda_F(u).
\end{align*}
The proof is finished.
\end{proof}
\begin{thm}\label{thm2}
Assume that Condition (A) holds. Then, there exists $c>0$ such that
\begin{align*}
\mathcal{T}(u)\leq \Lambda(u) \leq c \mathcal{T}(u),
\qquad \textrm{for any} \quad u > 1,
\end{align*}
that is, $\Lambda \asymp \mathcal{T}$ at $\infty$.
\end{thm}
\begin{proof}
We always have
\begin{align*}
\Lambda_F(u) \geq \Lambda(u) \geq \mathcal{T}(u), \qquad \textrm{for any} \quad
u > 1.
\end{align*}
Assuming that Condition (A) holds, we apply Proposition \ref{pro1} and obtain:
\begin{align*}
\Lambda_F(u)\leq c\mathcal{T}(u), \qquad \textrm{for any} \quad u > 1.
\end{align*}
The proof is finished.
\end{proof}
\begin{pro}\label{pro3}
Let $\phi, \upsilon: \mathbb{R}_{+} \to \mathbb{R}_{+}$ be two
continuous strictly monotone functions such that $\phi(k) \asymp
\sigma(k)$ and $\upsilon(k)\asymp |G_k|$ at $\infty$.
Then, under Condition (A):
\begin{align*}
\Lambda \stackrel{d}{\asymp} \phi \circ \upsilon^{-1} \qquad \textrm{ and } \qquad \mathcal{T} \stackrel{d}{\asymp} \phi \circ \upsilon^{-1} \qquad \textrm{at} \quad \infty.
\end{align*}
\end{pro}
\begin{proof}
Since $\Lambda \asymp \mathcal{T}$ we work with
$\mathcal{T}$. According to Proposition \ref{pro2}(4):
\begin{align*}
\frac{1}{2} \sigma(k+1) < \mathcal{T}(u) < \sigma(k), \qquad \textrm{for} \qquad
|G_k|\leq u< |G_{k+1}|.
\end{align*}
Hence for $u,k$ as above and for some $b_1,b_2>0$,
\begin{align*}
k \leq \upsilon^{-1}\left(\frac{u}{b_1}\right), \qquad k+1>\upsilon^{-1}\left(\frac{u}{b_2}\right).
\end{align*}
It follows that for some $b_3, b_4>0$
\begin{align*}
\mathcal{T}(u)<& \sigma(k)\leq (1+\lambda)\sigma(k+1)\leq b_3 (1+\lambda)\phi(k+1)\leq\\
\leq & b_3 (1+\lambda) \phi\circ\upsilon^{-1}\left(\frac{u}{b_2}\right),
\end{align*}
and
\begin{align*}
\mathcal{T}(u) >& \frac{1}{2} \sigma(k+1) \geq
\frac{1}{2(1+\lambda)}\sigma(k) \geq
\frac{b_4}{2(1+\lambda)}\phi(k)\geq \\
\geq& \frac{b_4}{ 2(1+\lambda)}
\phi\circ\upsilon^{-1}\left(\frac{u}{b_1}\right).
\end{align*}
Notice that the constants $b_1, b_2, b_3, b_4>0$ come from the relations
$\phi(k) \asymp~\sigma(k)$ and $\upsilon(k) \asymp |G_k|$. The proof is
finished.
\end{proof}
\begin{rem}\label{rem1}
If $\upsilon(k) =|G_k|$, we obtain the more precise
relations:
\begin{align}
\Lambda \asymp  \phi \circ \upsilon^{-1} \qquad \textrm{ and } \qquad \mathcal{T} \asymp \phi \circ \upsilon^{-1} \qquad \textrm{at} \quad \infty.
\end{align}
\end{rem}
\begin{exa}\label{ex1}$ $\\
1) Let $G=\bigcup_{n=0}^{\infty}G_n$, $G_n=Z(2)^n$. Then
$|G_n|=2^n$ and we can choose $\upsilon(x)=2^x$, $x\geq 0$.
Proposition \ref{pro3} and Remark \ref{rem1} yield the following result: Under
Condition (A), for any $\phi \asymp \sigma$ at $\infty$
\begin{align*}
\Lambda(\tau) \asymp \phi \circ \log (\tau^{\gamma}) \qquad \textrm{at} \quad \infty,
\qquad \textrm{where} \quad \gamma=1/\log 2.
\end{align*}
In particular, we have:\\
(1.1) If $c_k \asymp q^k$, $0<q<1$, then $\sigma(k) \asymp
q^k$, $ k \in \mathbb{Z}_{+}$. Hence,
\begin{align*}
\Lambda(\tau) \asymp \tau^{-A\gamma}  \qquad \textrm{at} \quad \infty, \quad \textrm{where} \quad A=\log \frac{1}{q}. 
\end{align*}
(1.2) If $c_k \asymp k^{-p}$, $p>1$, then $\sigma(k) \asymp
k^{-p+1}$, $ k \in \mathbb{Z}_{+}$. Hence,
\begin{align*}
\Lambda(\tau) \asymp (\log \tau)^{1-p} \quad \textrm{at} \quad \infty.
\end{align*}
(1.3) If  $c_k \asymp 1/(k\cdot \log k \cdot \log_{(2)}k \cdot \ldots \cdot (\log_{(n)}k)^p)$, $p>1$, then 
\begin{align*}
\sigma(k) \asymp \left( \log_{(n)}k\right)^{1-p}, \quad  k \in \mathbb{Z}_{+}.
\end{align*}
Hence, 
\begin{align*}
\Lambda(\tau) \asymp \left(\log_{(n+1)}
\tau\right)^{1-p} \quad \textrm{at} \quad \infty.
\end{align*}
2) Let $G=\bigcup_{n=0}^{\infty}G_n$, $G_n=S_n$, be the infinite
symmetric group, i.e. the group of all finite permutations of the
set $\mathbb{N}=\{1,2,\ldots\}$. We have $|G_n|=n!$. Hence we can
choose
$\upsilon (x) =\Gamma(x+1)$, where $\Gamma(x)$ is the gamma function
\begin{align}
\Gamma(x)=\int_0^{\infty} t^{x-1}e^{-t}dt.
\end{align}
By Stirling's formula \cite[(9.15)]{Fel}:
\begin{align}
\Gamma(1+x)=\sqrt{2 \pi x}\left(\frac{x}{e}\right)^x \left( 1+O\left(\frac{1}{x}\right)\right).
\end{align}
Let $\phi: \mathbb{R}_{+} \to \mathbb{R}_{+}$ be any continuous decreasing function such that $\phi(k)\asymp \sigma(k)$ at $\infty$.
Proposition \ref{pro3} and Remark \ref{rem1} show that under
Condition (A),
\begin{align}\label{equat6.8}
\Lambda \asymp \phi \circ \Gamma^{-1} \quad \textrm{at} \quad \infty.
\end{align}
\end{exa}
The upper order $\rho^{*}(f)$ and lower order $\rho_{*}(f)$ of a positive function $f$ are defined as
\begin{align*}
\rho^{*}(f):=\limsup_{x\to \infty} \frac{\log f(x)}{\log x}, \qquad \rho_{*}(f):=\liminf_{x\to \infty} \frac{\log f(x)}{\log x}.
\end{align*}
If $\rho^{*}(f)=\rho_{*}(f)< \infty$, we say that $f$ is of finite order $\rho(f):=\rho^{*}(f)$. 
\begin{pro}\label{pro5.12} Let $G=S_{\infty}$ be the infinite symmetric group.
$1.$~Under Condition (A), $\rho(\Lambda)=0$
(compare with Example \ref{ex1}(1.1)).
$2.$ Assume that $\sigma(k)\asymp|G_k|^{-\gamma}$ for some $\gamma>0$ (Condition (A) does not hold!). Then $\rho(\Lambda)=-\gamma$. In particular, for any $\epsilon>0$ there exist $c_1,c_2>0$ such that
\begin{align}
c_1u^{-(\gamma+\epsilon)}\leq \Lambda(u)\leq c_2u^{-(\gamma-\epsilon)} \qquad \textrm{at } \quad \infty.
\end{align}
\end{pro}
\begin{proof}[Proof of (1):]
Let $\sigma: \mathbb{R}_{+} \to \mathbb{R}_{+}$ be any continuous decreasing extension of the function  $\sigma: \mathbb{Z}_{+} \to \mathbb{R}_{+}$. By Proposition \ref{pro3} and Remark \ref{rem1},
\begin{align*}
\Lambda \asymp \sigma \circ \Gamma^{-1}= (\sigma \circ
\log)\circ (\exp \circ \Gamma^{-1})=(\sigma \circ \log)\circ (\Gamma
\circ \log)^{-1}.
\end{align*}
By Proposition \ref{pro4}, the function $\sigma \circ \log$ is
doubling. Hence for some $k>0$ and $u>>1$,
\begin{align*}
\sigma \circ \log(u) \geq u^{-k} \quad \textrm{at} \quad \infty.
\end{align*}
By Stirling's formula, for $u>>1$,
\begin{align*}
(\Gamma \circ \log(u)) \geq (\log u)^{\frac{1}{2}\log u},
\end{align*}
and
\begin{align*}
(\Gamma \circ \log)^{-1}(u) \leq \exp\left( \frac{3 \log u}{\log
\log u} \right).
\end{align*}
It follows that for $u>>1$,
\begin{align*}
0\geq \log \left( \Lambda(u) \right) \geq \frac{-3 k \log
u}{\log \log u}.
\end{align*}
Evidently this inequality implies that $\rho(\Lambda)=0$.

\textit{Proof of (2):}
By Theorem \ref{thm5.1} and Proposition \ref{proposit6.4},
\begin{align*}
\frac{1}{\Lambda_F}\leq \frac{1}{\Lambda}\leq \frac{1}{\mathcal{T}}.
\end{align*}
It follows that
\begin{align*}
\rho_{*}\left(\frac{1}{\Lambda_F}\right)\leq \rho_{*}\left(\frac{1}{\Lambda}\right)\leq \rho^{*}\left(\frac{1}{\Lambda}\right)\leq \rho^{*}\left( \frac{1}{\mathcal{T}}\right).
\end{align*}
For $|G_k|\leq u< |G_{k+1}|$, by Proposition \ref{pro2},
\begin{align*}
\frac{1}{2}\sigma(k+1)< \mathcal{T}(u)< \sigma(k).
\end{align*}
Hence for such $u$ and $k$, and for some $c_1>0$,
\begin{align*}
\frac{1}{\mathcal{T}(u)}\leq \frac{2}{\sigma(k+1)}\leq c_1 |G_{k+1}|^{\gamma}=c_1 ((k+1)!)^{\gamma},
\end{align*}
and 
\begin{align*}
\log \frac{1}{\Lambda(u)}\leq \log \frac{1}{\mathcal{T}(u)}\leq c_1+\gamma \log (k+1)! \leq c_1+\gamma \log (k+1) + \gamma \log u.
\end{align*}
By Stirling's formula $\log k! \sim k\log k$, hence for any $\epsilon>0$ there exists $u_0>1$ such that for all $u > u_0$,
\begin{align*}
\log \frac{1}{\Lambda(u)}\leq (\gamma + \epsilon) \log u.
\end{align*}
This evidently yields the inequality
\begin{align}\label{equat6.10}
\rho^{*}\left(\frac{1}{\Lambda}\right)\leq \gamma.
\end{align}
By the definition of the function $\Lambda_F$, see (\ref{6.4}),
\begin{align*}
\frac{1}{\Lambda_F(u)}=(F^{-1}(u)-1)^2.
\end{align*}
It follows that 
\begin{align*}
\rho_{*}\left( \frac{1}{\Lambda_F(u)}\right)=2 \rho_{*}(F^{-1})=\frac{2}{\rho^{*}(F)}.
\end{align*}
For $n \leq u \leq n+1$,
\begin{align*}
F(u)\leq F(n+1)=|G_{k(n+1)}|=(k(n+1))!,
\end{align*}
where by (\ref{6.3}),
\begin{align*}
k(n+1):=\min \{ k: \, \lambda_1(G_k)\leq (n+1)^{-2}\}.
\end{align*}
By Proposition \ref{proposit6.3},
\begin{align*}
\frac{1}{2}\sigma(k) < \lambda_1(G_k)< \sigma(k),
\end{align*}
hence, for some $c_2>0$,
\begin{align*}
k(n+1) &\leq \min \{ k: \, \sigma(k)\leq (n+1)^{-2}\}\leq \\
& \leq \min \{ k: \, |G_k|\geq c_2(n+1)^{2/\gamma}\}:=\bar{k}.
\end{align*}
Evidently, we have
\begin{align*}
(\bar{k}-1)!\leq c_2 (n+1)^{2/\gamma}\leq \bar{k}!=\bar{k}(\bar{k}-1)!.
\end{align*}
It follows that
\begin{align*}
F(u)\leq \bar{k}!\leq c_2 \bar{k}(n+1)^{2/\gamma}.
\end{align*}
By Stirling's formula, for some $c_3$, $c_4>0$,
\begin{align*}
 \bar{k}\leq \bar{k} \log \bar{k}\leq c_3 \log (\bar{k}-1)!\leq c_3 \log c_2 (n+1)^{2/\gamma}\leq c_4 \log (n+1).
\end{align*}
Hence for some $c_5>0$ and any $\epsilon >0$, and $u>>1$ we obtain
\begin{align*}
\log F(u) \leq c_5 +\log \log (n+1) +\frac{2}{\gamma}\log (n+1)\leq \left(\frac{2}{\gamma}+\epsilon \right) \log u.
\end{align*}
This evidently yields
\begin{align*}
\rho^{*}(F)\leq \frac{2}{\gamma}
\end{align*}
and
\begin{align}\label{equat6.11}
\rho_{*}\left(\frac{1}{\Lambda} \right)=\frac{2}{\rho^{*}(F)}\geq \gamma.
\end{align}
 The inequalities (\ref{equat6.10}) and (\ref{equat6.11}) prove the claim.
\end{proof}
Let $\phi: \mathbb{R}_{+} \to \mathbb{R}_{+}$ be any continuous decreasing function such that $\phi(k)\asymp \sigma(k)$ at $\infty$.
The following three examples illustrate Proposition~\ref{pro5.12}(1):\\
(2.1) If $c_k\asymp q^k$, $0<q<1$, then $\sigma(k) \asymp q^k$, $k \in \mathbb{Z}_{+}$.
By (\ref{equat6.8}),
\begin{align*}
\Lambda(\tau) \asymp(\phi \circ \log)\circ (\Gamma \circ
\log)^{-1}(\tau) \asymp \left( (\Gamma \circ
\log)^{-1}(\tau)\right)^{-A},
\end{align*}
where $A=\log \frac{1}{q}$.
By Stirling's formula,
\begin{align*}
\log \circ \left(\Gamma \circ \log\right)^{-1}(\tau)=\Gamma^{-1}(\tau) \sim
\frac{\log \tau}{\log \log \tau} \quad \textrm{at} \quad \infty.
\end{align*}
Hence,
\begin{align*}
\Lambda(\tau) =\exp \left( -\log \frac{1}{\Lambda(\tau)}\right),
\end{align*}
and
\begin{align*}
\log \frac{1}{\Lambda(\tau)} \sim \frac{A\log \tau}{\log \log
\tau} \quad \textrm{at} \quad \infty.
\end{align*}
(2.2) If $c_k \asymp k^{-p} $, $p>1$, then $\sigma(k) \asymp
k^{-p+1}$, $k \in \mathbb{Z}_{+}$, and
\begin{align*}
\Lambda(\tau) \asymp \left(\phi \circ \Gamma^{-1}\right)(\tau) \asymp
\left(\frac{\log \tau}{\log \log \tau}\right)^{-p+1} \quad
\textrm{at} \quad \infty.
\end{align*}
(2.3) If  $c_k \asymp 1/(k \cdot\log k \cdot \log_{(2)}k \cdot \ldots \cdot (\log_{(n)}k)^p)$, $p>1$, then 
\begin{align*}
\sigma(k) \asymp(\log_{(n)}k)^{1-p}, \qquad k \in \mathbb{Z}_{+},
\end{align*}
and
\begin{align*}
\Lambda(\tau) \asymp \left(\phi \circ \Gamma^{-1}\right)(\tau) \asymp
\left(\log_{(n+1)} \tau \right)^{1-p} \quad \textrm{at} \quad \infty.
\end{align*}
\begin{rem}
Observe that in all examples (2.1)-(2.3), the function $\Lambda$ is comparable to a slowly varying function $\alpha$. For example, in (2.1) one can take $\alpha(\tau)=((\Gamma \circ \log)^{-1}(\tau))^{-A}$. Recall that slow variation means that
\begin{align*}
\lim_{\tau \to \infty} \frac{\alpha(\lambda \tau)}{\alpha(\tau)}=1 \qquad \textrm{for any } \lambda>0.
\end{align*}
\end{rem}
Evidently this property is stronger than the condition $\rho(\Lambda)=0$ in Proposition \ref{pro5.12}(1). See \cite[Sec 1.2, 2.4 and Thm. 2.4.7]{BGT}.

\section{Spectral distribution and return probability.}

Let $\lambda \to E_{\lambda}$ be the spectral resolution of the
Laplacian $\Delta=P-I$,
\begin{align*}
-\Delta=\int_0^{\infty} \lambda dE_{\lambda}.
\end{align*}
We define the spectral distribution function $\lambda \to
N(\lambda)$ as follows:
\begin{align*}
N(\lambda):=(E_{\lambda} \delta_e, \delta_e).
\end{align*}
Evidently, $\lambda \to N(\lambda)$ is a right-continuous, non-decreasing step-function. It has jumps at the
points $\lambda_k=\sigma(k)$ and $N(\lambda_k)=1/|G_{k}|$, $ k \in \mathbb{N}$.
Indeed, by the definition of $E_{\lambda}$, we must have
\begin{align*}
N(\lambda_k)=&(E_{\sigma(k)}\delta_e,
\delta_e)=(P_{k}\delta_e, \delta_e)=\\
=& (\delta_e *m_{k},
\delta_e)=m_{k}(\{e\})=\frac{1}{|G_{k}|}.
\end{align*}

On finitely generated groups, and for symmetric probability measures with generating supports and finite second moments, the (dilatational equivalence class of the) function $N$ is stable under quasi-isometry. See \cite{BPSa}, \cite{K1}, \cite{K2}.  Under mild regularity assumptions, $\Lambda$ and $N$ are related by the formula
\begin{align*}
N(\lambda ) \stackrel{d}{\simeq} \frac{1}{\Lambda^{-1}(\lambda)}.
\end{align*}
See \cite{BPSa} and Proposition \ref{pro5} below. 

In the second part of this section we will use our computations of $N$ to evaluate  the return probability function $t\to p(t)$. This is a first step in estimating the transition function/the heat kernel of the random walk on $G$ under consideration. Heat kernel bounds will be presented in the final Section 8.

\begin{pro} \label{pro5}
Under Condition (A), the following properties hold.
\begin{align}\label{equat7.1}
N \stackrel{d}{\simeq} \frac{1}{\mathcal{T}^{-1}}, \qquad \textrm{at} \quad 0.
\end{align}
\begin{align}\label{equat7.2}
\Lambda \asymp \mathcal{T}, \qquad \textrm{at} \quad \infty. 
\end{align}
\end{pro}
\begin{proof} The equivalence
(\ref{equat7.2}) follows from Theorem \ref{thm2}. To prove (\ref{equat7.1}), observe that for $\sigma(k+1) \leq u < \sigma(k)$,
\begin{align*}
{N(u)}=\frac{1}{|G_{k+1}|}.
\end{align*}
Also, by Proposition \ref{pro2}(4), for $|G_k| \leq \tau \leq
|G_{k+1}|$,
\begin{align*}
\frac{1}{2}\sigma(k+1)< \mathcal{T}(\tau) < \sigma(k).
\end{align*}
Define $\theta_1:=\mathcal{T}(|G_{k+1}|)$ and
$\theta_2:=\mathcal{T}(|G_k|)$. Then,
\begin{align*}
\frac{1}{2}\sigma(k+1)<\theta_1<\theta_2<\sigma(k).
\end{align*}
Observe that Condition (A) implies that there exists a constant $\lambda>0$
such that
\begin{align*}
\sigma(k+1)> \frac{1}{1+\lambda} \sigma(k), \qquad k=0,1,\ldots
\,.
\end{align*}
Putting all these facts together and the fact that $\theta \to
\mathcal{T}^{-1}(\theta)$ decreases, we obtain that for $\sigma(k+1)
\leq u < \sigma(k)$,
\begin{align*}
\frac{1}{N(u)}=&|G_{k+1}|=\mathcal{T}^{-1}(\theta_1)>
\mathcal{T}^{-1}(\sigma(k))>\mathcal{T}^{-1}((1+\lambda )\sigma(k+1) ) >\\
>& \mathcal{T}^{-1}((1+\lambda )u),
\end{align*}
and
\begin{align*}
\frac{1}{N(u)} =&|G_{k+1}|=\mathcal{T}^{-1}(\theta_1)<
\mathcal{T}^{-1}\left(\frac{1}{2}\sigma(k+1)\right)<\\
<&
\mathcal{T}^{-1}\left(\frac{1}{2(1+\lambda )}\sigma(k)\right)<
 \mathcal{T}^{-1}\left(\frac{u}{2(1+\lambda)}\right).
\end{align*}
The proof is finished.
\end{proof}
\begin{exa}\label{ex2} $\textrm{ }$\\
Let $G=\bigcup_{n=0}^{\infty}G_n$ and $\upsilon:\mathbb{R}_{+} \to
\mathbb{R}_{+} $ be the volume function, that is
$\upsilon(k)=|G_k|$, $k=0,1,\ldots\,$. Let $\phi: \mathbb{R}_{+} \to \mathbb{R}_{+}$ be any continuous decreasing function such that $\phi(k)\asymp \sigma(k)$ at $\infty$. Then, according to Proposition \ref{pro3}
 and Proposition \ref{pro5}, under Condition (A),
\begin{align*}
N \stackrel{d}{\simeq} \frac{1}{\upsilon \circ
\phi^{-1}}.
\end{align*}
1) Let $G=\mathbb{Z}(2)^{(\infty)}$. Then, $\upsilon(x)=2^x$, and the
formula for $N$ takes the following form:
\begin{align*}
N(u) \stackrel{d}{\asymp} \exp(-\frac{1}{\gamma}\phi^{-1}(u)),
\qquad \textrm{where} \quad \gamma=1/\log 2.
\end{align*}
In particular, we obtain the following estimates.\\
(1.1) If $c_k \asymp q^k$, $0<q<1$, then $\sigma(k) \asymp q^k$, $k \in \mathbb{Z}_{+}$, and, by Remark~\ref{rem1},
\begin{align*}
N(u) \asymp u^{1/A\gamma}\quad \textrm{at } \quad 0, \quad \textrm{where} \quad A=\log\frac{1}{q} \quad \textrm{and} \quad \gamma=\frac{1}{\log 2}.
\end{align*}
(1.2) If $c_k \asymp k^{-p}$, $p>1$, then $\sigma(k) \asymp k^{-p+1}$, $k \in \mathbb{Z}_{+}$,  and
\begin{align*}
N(u) \stackrel{d}{\simeq} \exp\{-u^{\frac{1}{1-p}}\} \qquad
\textrm{at } \quad 0  .
\end{align*}
(1.3) If  $c_k \asymp 1/(k \cdot \log k \cdot \log_{(2)}k \cdot \ldots \cdot (\log_{(n)}k)^p)$, $p>1$, then
\begin{align*}
\sigma(k) \asymp (\log_{(n)}k)^{1-p}, \qquad k \in \mathbb{Z}_{+}
\end{align*}
and
\begin{align*}
N(u) \stackrel{d}{\simeq}
\exp\{-\exp_{(n)}\left(u^{\frac{1}{1-p}}\right)\} \qquad \textrm{at } \quad 0
.
\end{align*}
2) Let $G=S_{\infty}$ be the infinite symmetric
group endowed with its volume function $\upsilon(x)=\Gamma(1+x)$, that is $G_n=S_n$ and $|G_n|=n!$. 

(a) Assume that Condition (A) holds. Then the formula for
the function $N$ takes the following form:
\begin{align*}
\log \frac{1}{N(u)}  \stackrel{d}{\simeq} (\log \Gamma)\circ
\phi^{-1}(u).
\end{align*}
Stirling's formula and straightforward computations give the following results.\\
(2.1) If $c_k \asymp q^k$, $0<q<1$, then $\sigma(k) \asymp q^k$, $k \in \mathbb{Z}_{+}$, and with $A=\log \frac{1}{q}$
\begin{align*}
\log \frac{1}{N(u)} \sim \frac{1}{A} \left(\log \frac{1}{u}\right)\left(
\log \log \frac{1}{u}
  \right)\quad \textrm{at} \quad 0.
\end{align*}
(2.2) If $c_k \asymp k^{-p}$, $p>1$, then $\sigma(k)\asymp k^{-p+1}$, $k \in \mathbb{Z}_{+}$, and
\begin{align*}
N(u) \stackrel{d}{\simeq} \exp \left(- \left(\frac{1}{u}\right)^{\frac{1}{p-1}} \log \frac{1}{u}\right)
\quad \textrm{at} \quad 0 .
\end{align*}
(2.3) If $c_k \asymp 1/(k \cdot\log k \cdot \log_{(2)}k \cdot \ldots \cdot (\log_{(n)}k)^p)$, $p>1$, then
\begin{align*}
\sigma(k) \asymp (\log_{(n)}k)^{-p+1}, \qquad k \in \mathbb{Z}_{+}
\end{align*}
and
\begin{align*}
N(u) \stackrel{d}{\simeq}  \exp \left( -\exp_{(n)}\left(\frac{1}{u}\right)^{\frac{1}{p-1}}\right) \quad
\textrm{at}\quad 0 .
\end{align*}

(b) Assume that $\sigma(k) \asymp |G_k|^{-\gamma}$ for some $\gamma>0$. \textit{We claim that in this case $N(u)$ is of finite order $1/\gamma$ at zero, that is, $u\to 1/N(1/u)$ is of finite order $1/\gamma$ at infinity. In other words, for any $\epsilon >0$ there exist $c_1,c_2>0$ such that}
\begin{align*}
c_1u^{\frac{1}{\gamma}+\epsilon} \leq N(u) \leq  c_2u^{\frac{1}{\gamma}} \qquad \textrm{at}\quad 0.
\end{align*}
Indeed, let $\sigma(k+1)\leq u < \sigma(k) $. By assumption, for some $c_1>0$,
 \begin{align*}
N(u)=\frac{1}{|G_{k+1}|}=\frac{1}{(k+1)!}\leq \left(\frac{u}{c_1}\right)^{1/\gamma}.
\end{align*}
It follows that
\begin{align*}
\rho_{*}(N):= \liminf_{\lambda \to \infty} \frac{\log \frac{1}{N(1/\lambda )}}{\log \lambda} \geq \frac{1}{\gamma}.
\end{align*}
On the other hand, by assumption,
\begin{align*}
N(u)=\frac{1}{(k+1)!}=\frac{(k!)^{-1}}{k+1}\geq \left(\frac{u}{c_2}\right)^{1/\gamma} \frac{1}{k+1}.
\end{align*}
 By Stirling's formula, for any $\epsilon>0$ and for all $k\geq k(\epsilon)>1$,
\begin{align*}
k+1\leq k \log k \leq (1+\epsilon) \log k! \leq \frac{1+\epsilon}{\gamma} \log \frac{c_2}{u}\leq \frac{1+2\epsilon}{\gamma}\log \frac{1}{u}.
\end{align*} 
It follows that for some $c_3>0$,
\begin{align*}
N(u)\geq \frac{c_3 u^{1/\gamma}}{\log \frac{1}{u}}.
\end{align*}
This inequality shows that
\begin{align*}
\rho^{*}(N):= \limsup_{\lambda \to \infty} \frac{\log \frac{1}{N(1/\lambda )}}{\log \lambda} \leq \frac{1}{\gamma}.
\end{align*}
The claim is proved.
$\textrm{ }$ \hfill{$\square$}
\end{exa}

Let $(\mu_t)_{t>0}$ be a weakly continuous convolution semigroup of probability measures on $G$ associated with the measure $\mu=\mu(c)$, that is, $\mu_t|_{t=1}=\mu$ (see Proposition \ref{proposit3.2}). Let $P$ and $\Delta=P-I$ be the corresponding transition operator and the corresponding Laplacian.

We define  $p: \mathbb{R}_{+} \to  \mathbb{R}_{+}$ as follows:
\begin{align*}
p(t):=(\delta_e*\mu_t, \delta_e)=(P^t\delta_e,\delta_e),
\end{align*}
and call this function \textit{the return probability function}. It coincides with the probability of return at time $t>0$ to the identity of the continuous-time process $X(t)$ defined by the semigroup $(\mu_t)_{t>0}$.

Let $(E_{\lambda})$ be the spectral resolution associated with $-\Delta$. Then equation (\ref{equat3.3}) and the spectral theory show that
\begin{align*}
 p(t)=((I+\Delta)^t\delta_e, \delta_e)=\int_0^{\infty} (1 - \lambda)^t dN(\lambda).
\end{align*}
Our first observation is that
\begin{align}\label{equat7.3}
p(t)  \stackrel{d}{\simeq}\int_0^{\infty} e^{-\lambda t}dN(\lambda).
\end{align}
This equivalence relation follows from the fact that the measure defined as  $B \to \int_B dN$ is concentrated on the interval $ [0, \sigma(0)] \subset [0,1]$ and that for all $\lambda \in [0, \sigma(0)]$ and $t>0$, 
\begin{align*}
e^{-\delta \lambda t}\leq (1 - \lambda)^t \leq  e^{-\lambda t}, \qquad \textrm{for some} \quad \delta>1. 
\end{align*}
Writing the function $t \to p(t)$ in the form
\begin{align*}
 p(t)=\exp (-t \cdot R(t)), \quad t>0,
\end{align*}
we observe that since the group $G$ is amenable, $R(t)=o(1)$ at $\infty$.
\begin{pro}\label{proposit7.3}
Let $\upsilon, \phi: \mathbb{R}_{+} \to \mathbb{R}_{+}  $ be two continuous monotone functions such that $\upsilon(k)\asymp|G_k|$ and $\phi(k) \asymp \sigma(k)$, $k\in \mathbb{Z}_{+}$. Then, under Condition (A),
\begin{align*}
 R \stackrel{d}{\asymp} \left(\frac{(\log \upsilon) \circ \phi^{-1}}{id}\right)^{-1} .
\end{align*}
\end{pro}
\begin{proof}
Using first Proposition \ref{thm1}, and then Proposition \ref{pro5} and Proposition \ref{pro3}, we can write
\begin{align*}
\log \frac{1}{p} \sim \mathcal{L} \left(\log \frac{1}{N}\right) \stackrel{d}{\simeq} \mathcal{L}((\log \upsilon) \circ \phi^{-1})\stackrel{d}{\asymp} id\cdot \left(\frac{(\log \upsilon) \circ \phi^{-1}}{id}\right)^{-1} .
\end{align*}
This evidently gives the desired result.
\end{proof}
\begin{exa} $\textrm{ }$\\ \label{ex3}
1) Let $G=\mathbb{Z}(2)^{(\infty)}$. Choosing $\upsilon(x)=2^x$, the
formula for $p$ from Proposition \ref{proposit7.3} takes the following form:
\begin{align*} 
p(t)=\exp(-tR(t)), \quad R(t) \stackrel{d}{\asymp} \left(\frac{\phi^{-1}}{id} \right)^{-1}.
\end{align*}
In particular we obtain the following estimates.\\
(1.1) If $c_k \asymp q^k$, $0<q<1$, then $\sigma(k) \asymp q^k$ and $N(\lambda) \asymp \lambda^{1/A\gamma}$ at~$0$, where $A=\log \frac{1}{q}$ and $\gamma=1/ \log 2$ (see Example \ref{ex2}(1)). Hence, a standard Laplace transform argument gives
\begin{align*}
p(t) \stackrel{d}{\simeq} \int_0^{\infty} e^{-\lambda t} dN(\lambda) \asymp t^{-1/A\gamma} \quad \textrm{at }\quad \infty.
\end{align*}
(1.2) If $c_k \asymp k^{-p}$, $p>1$, then $\sigma(k) \asymp k^{-p+1}$ at $\infty$, and
\begin{align*}
R(t) \asymp t^{-\frac{p-1}{p}},\quad p(t)\stackrel{d}{\simeq} \exp\left(-t^{\frac{1}{p}}\right) \quad
\textrm{at }\quad \infty.
\end{align*}
(1.3) If $c_k \asymp 1/(k \cdot\log k \cdot \log_{(2)}k \cdot \ldots \cdot (\log_{(n)}k)^p)$ and $p>1$, then
\begin{align*}
\sigma(k)\asymp (\log_{(n)}k)^{-p+1}\quad \textrm{at} \quad \infty. 
\end{align*}
\textit{We claim that in this case $R \asymp \phi$.} More generally, the following proposition holds true.
\begin{pro} \label{propo6.2}
Assume that $x \to (\phi \circ \exp)(x)$ is doubling, then
\begin{align*}
\left( \frac{\phi^{-1}}{id}  \right)^{-1} \asymp \phi.
\end{align*}
\end{pro}
\begin{proof}
Write
\begin{align*}
\frac{\phi^{-1}(x)}{x}=\exp \left( (\phi \circ \exp)^{-1}(x)+\log \frac{1}{x}\right).
\end{align*}
Since $\phi \circ \exp $ is doubling, there exists $d>0$, such that $(\phi \circ \exp )(x) \geq x^{-d}$ at $\infty$. It follows that at $0$, $(\phi \circ \exp)^{-1}(x) \geq x^{-1/d}$. Hence, for some $A>1$, and $x$ close to $0$,
\begin{align*}
\phi^{-1}(x)\leq \frac{\phi^{-1}(x)}{x}\leq \exp \left(A (\phi \circ \exp\right)^{-1}(x)).
\end{align*}
This implies, that at $\infty$,
\begin{align*}
\phi(x) \leq \left(\frac{\phi^{-1}}{id}\right)^{-1}(x)\leq (\phi \circ \exp)\left(\frac{1}{A} \log x\right).
\end{align*}
Again, using the fact that $\phi \circ \exp$ is doubling we obtain the desired result. 
\end{proof}
Finally, assuming (1.3), we obtain
\begin{align*}
p(t)\stackrel{d}{\simeq} \exp \left(-\frac{t}{(\log_{(n)}t)^{p-1}} \right) \quad \textrm{ at } \quad \infty.
\end{align*}
2) Let $G=S_{\infty}$. In this case $\upsilon(x)=\Gamma(1+x)$, $\log \upsilon(x) \sim x \log x$ and, assuming that Condition (A) holds, the formula for
$p(t)$ from Proposition \ref{proposit7.3} takes the following form:
\begin{align*}
p(t)=\exp(-tR(t)), \quad R(t)  \stackrel{d}{\asymp} \left(\frac{\phi^{-1}\log \phi^{-1}}{id} \right)^{-1}.
\end{align*}
(2.1) If $c_k \asymp q^k$, $0<q<1$, then $\sigma (k) \asymp q^k$, $k \in \mathbb{Z}_{+}$. Example  \ref{ex2} (2.1) and Proposition \ref{thm1} (1) yield:
\begin{align*}
\log \frac{1}{N(u)} \sim \frac{1}{A}\log \frac{1}{u} \log \log \frac{1}{u}\qquad \textrm{ at } \quad 0, 
\end{align*}
and
\begin{align*}
\log \frac{1}{p(t)}\sim \frac{1}{A} (\log t)(\log \log t) \qquad \textrm{ at } \quad \infty.
\end{align*}
(2.2) If $c_k \asymp k^{-p}$ and $p>1$, then $\sigma (k) \asymp k^{-p+1}$ at $\infty$. This implies
\begin{align*}
R(t) \asymp \left( \frac{\log t}{t}\right)^{1 - \frac{1}{p}}\quad \textrm{and} \quad p(t)\stackrel{d}{\simeq} \exp \left( -t^{\frac{1}{p}}( \log t)^{1-\frac{1}{p}} \right) \quad \textrm{ at } \quad \infty.
\end{align*}
(2.3) Let $c_k \asymp k^{-1}/(\log k \cdot \log_{(2)}k \cdot \ldots \cdot (\log_{(n)}k)^p)$, $p>1$, then 
\begin{align*}
\sigma (k)\asymp (\log_{(n)}k)^{-p+1} \qquad \textrm{at} \quad \infty.
\end{align*}
This case is similar to (1.3):
proceeding as in the proof of Proposition \ref{propo6.2}, we obtain $R\stackrel{d}{\asymp}\phi$. Hence,
\begin{align*}
R(t)  \asymp \phi(t) = (\log_{(n)}t)^{-p+1} \quad \textrm{ at } \quad \infty,
\end{align*}
and
\begin{align*}
p(t)  \stackrel{d}{\simeq} \exp \left(- \frac{t}{(\log_{(n)}t)^{p-1}}\right) \quad \textrm{ at } \quad \infty.
\end{align*}
(2.4) Assume that $\sigma(k)\asymp |G_k|^{-\gamma}$ for some $\gamma>0$. \textit{Then the function $p(t)$ is of finite order $-1/\gamma$.} Indeed, $p(t)$ and $N(\lambda)$ are related by (\ref{equat7.3}). Example \ref{ex2}(2b) and a standard Laplace transform argument yield the result. In particular, for any $\epsilon>0$ there exist $c_1,c_2>0$ such that
\begin{align*} 
 c_2 t^{-\frac{1}{\gamma}-\epsilon}\leq p(t)\leq c_1 t^{-\frac{1}{\gamma}} \qquad \textrm{at} \quad \infty.
\end{align*}
\end{exa}
Some particular results based on the computations in Examples \ref{ex1}, \ref{ex2} and \ref{ex3} are presented in Table 2.

\begin{table}
\begin{center}
\begin{turn}{90}
\setlength\extrarowheight{10pt}
\begin{threeparttable}
\caption{\textbf{Computations of functions $\Lambda(\lambda)$, $N(\tau)$ and $p(t)$ in several examples.}}
\begin{tabular}{|p{6cm}|p{4cm}|p{6cm}|p{5.5cm}|}\hline
\multicolumn{4}{|c|}{ \rule[-0.2cm]{0pt}{0.7cm} Group $G=\bigcup_k G_k, \quad G_k=\mathbb{Z}(2)^{k}$}\\ \hline 
\rule[-0.2cm]{0pt}{0.6cm}$\sigma$ at $\infty$&$\Lambda$ at $\infty$ & $N$ at $0$ & $p$ at $\infty$\\
\hline 
$\sigma(k)\asymp |G_k|^{-\alpha}$, $\alpha>0$ &$\Lambda(\tau) \asymp \tau^{-\alpha}$  \rule[-0.3cm]{0pt}{0.9cm}&
$N(u)\asymp u^{1/\alpha}$ &$p(t)\asymp t^{-1/\alpha}$ \\ \hline
$\sigma(k)\asymp \left(\frac{1}{k}\right)^p$, $p>0$    &\rule[-0.3cm]{0pt}{0.9cm}$\Lambda(\tau) \asymp \frac{1}{(\log \tau)^{p}}$ &$N(u)\stackrel{d}{\simeq} \exp\left\{-\left(\frac{1}{u}\right)^{\frac{1}{p}}\right\}$ &$p(t)\stackrel{d}{\simeq} \exp\left(-t^{\frac{1}{p+1}}\right)$ \\
\hline 
$\sigma(k)\asymp \frac{1}{\left( \log_{(n)}k\right)^p}$, $n\geq 1$, $p>0$ \tnote{\textbf(a)}&\rule[-0.4cm]{0pt}{1cm}$\Lambda(\tau)\asymp\frac{1}{(\log_{(n+1)} \tau)^{p}}\,$ \tnote{\textbf(a)} &$N(u)\stackrel{d}{\simeq}
\exp\left\{-\exp_{(n)}\left(\frac{1}{u}\right)^{\frac{1}{p}}\right\}\, $ \tnote{\textbf(b)} &$p(t)\stackrel{d}{\simeq} \exp \left(-\frac{t}{(\log_{(n)}t)^{p}} \right)\,$ \tnote{\textbf(b)} \\
\hline \hline
\multicolumn{4}{|c|}{ \rule[-0.2cm]{0pt}{0.7cm} Group $G=\bigcup_k G_k, \quad G_k=S_{k}$}\\ \hline 
$\sigma(k)\asymp |G_k|^{-\gamma}$, $\gamma>0$& $\rho(\Lambda)=-\gamma \,$ \tnote{\textbf(c)}& $\rho(N)=\frac{1}{\gamma}\,$ \tnote{\textbf(c)}& \rule[-0.3cm]{0pt}{0.8cm}$\rho(p)=-\frac{1}{\gamma}\,$ \tnote{\textbf(c)} \\ \hline
$\sigma(k)\succ |G_k|^{-\gamma}$, $\forall \gamma>0$\rule[-0.3cm]{0pt}{0.8cm}& $\rho(\Lambda)=0$ & $\rho(N)=+\infty$& $\rho(p)=-\infty$\\ \hline
 $\sigma(k)\asymp q^k$, $0<q<1$ \rule[-0.3cm]{0pt}{0.8cm}& $\log \frac{1}{\Lambda(\tau)} \sim \frac{A \log \tau}{\log_{(2)} \tau}\,$ \tnote{\textbf(d)} &
$\log \frac{1}{N(u)}\sim \frac{1}{A} \left( \log \frac{1}{u}\right) \left( \log_{(2)} \frac{1}{u}\right)\,$\tnote{\textbf(d)} &$\log \frac{1}{p(t)}\sim \frac{1}{A}(\log t)( \log_{(2)}t)\,$ \tnote{\textbf(d)} \\ \hline
$\sigma(k)\asymp \left(\frac{1}{k}\right)^p$, $p>0$    &\rule[-0.3cm]{0pt}{0.9cm}$\Lambda(\tau) \asymp \left( \frac{\log_{(2)}\tau}{\log \tau}\right)^p$ &$N(u)\stackrel{d}{\simeq} \exp\left\{-\left(\frac{1}{u}\right)^{\frac{1}{p}}\log \frac{1}{u}\right\}$ &$p(t)\stackrel{d}{\simeq} \exp\left(-t^{\frac{1}{p+1}}\left( \log t \right)^{\frac{p}{1+p}}\right)$ \\
\hline 
$\sigma(k)\asymp \frac{1}{\left( \log_{(n)}k\right)^p}$, $n\geq 1$, $p>0$&\rule[-0.4cm]{0pt}{1cm}$\Lambda(\tau)\asymp\frac{1}{(\log_{(n+1)} \tau)^{p}} $ &$N(u)\stackrel{d}{\simeq}
\exp\left\{-\exp_{(n)}\left(\frac{1}{u}\right)^{\frac{1}{p}}\right\}$ &$p(t)\stackrel{d}{\simeq} \exp \left(-\frac{t}{(\log_{(n)}t)^{p}} \right)$ \\
\hline 
\end{tabular}
\begin{tablenotes}[para]
\item[\textbf(a)] $\log_{(k)}(t)=\underbrace{\log
\left(1+\log\left(1+...\log(1+t)\right)\right)}_{\textrm{k
times}}$
\item[\textbf(b)] $\exp_{(k)}(t)=\underbrace{\exp\left(\exp\left(...\exp(t)\right)\right)}_{\textrm{k
times}}$
\item[\textbf(c)] $\rho(f)$ is the order of function $f$ at $\infty$(resp. at $0$), see Example \ref{ex1}(2) (resp. \ref{ex2} (2b)). 
\item[\textbf(d)] $A=\log \frac{1}{q}$.
\end{tablenotes}
\end{threeparttable}
\end{turn}
\end{center}
\end{table}

\newpage

\section{Metric structure and heat kernel bounds.}\label{heat}

In Sections 6 and 7 we developed tools to compute the functions $\Lambda$, $N$ and $p$. These functions are closely related to the spectrum of the Laplacian $\Delta$. Our next step is to estimate the transition function/the heat kernel $h(t;x,y)$, associated with the random walk $X(n)$ on $G$ under consideration. For that, we will introduce a metric $\rho=\rho(x,y)$ on $G$ intrinsically associated to $X(n)$. We will show that as in the classical potential theory, $h(t;x,y)$ can be estimated in terms of the variables $t$ and $\rho$.

For any $x \in G$, put $n:=\min \{k \in \mathbb{N}:\, x \in G_k\}$ and define the $\sigma$-value $|x|_{\sigma}$ of $x$ as follows:
\begin{align*} 
|x|_{\sigma}:=\frac{1}{\sigma(n-1)}-1, \qquad \sigma(-1):=1.
\end{align*}
We define $\rho: G\times G \to \mathbb{R}_{+}$ by the following equation:
\begin{align*} 
\rho(x,y):=|x^{-1}y|_{\sigma}.
\end{align*}
Since $n \to \sigma(n)$ is decreasing, for any $a,b \in G$,
\begin{align*} 
|ab|_{\sigma} \leq \max \{ |a|_{\sigma}, |b|_{\sigma}\}.
\end{align*}
Hence $\rho $ is a metric on $G$ and $(G,\rho)$ is a complete (ultra-) metric space.
\begin{pro}\label{pro8.1} Any ball $B_r(a)\subset (G, \rho)$ is of the form
\begin{align*} 
B_r(a)=aG_k, \quad \textrm{ for some } k\geq 0,
\end{align*}
having radius
\begin{align*} 
r=\frac{1}{\sigma(k-1)}-1,
\end{align*}
and volume (with respect to the counting measure)
\begin{align*} 
|B_r(a)|=\frac{1}{N_{-}\left(\frac{1}{1+r}\right)},
\end{align*}
where $N_{-}(\lambda):=N(\lambda-)$ is the left-continuous modification of $N$.
\end{pro}
\noindent
The proof of Proposition \ref{pro8.1} is straightforward and follows by inspection.
\begin{figure}[h]
\begin{center}
\includegraphics[width=6cm]
{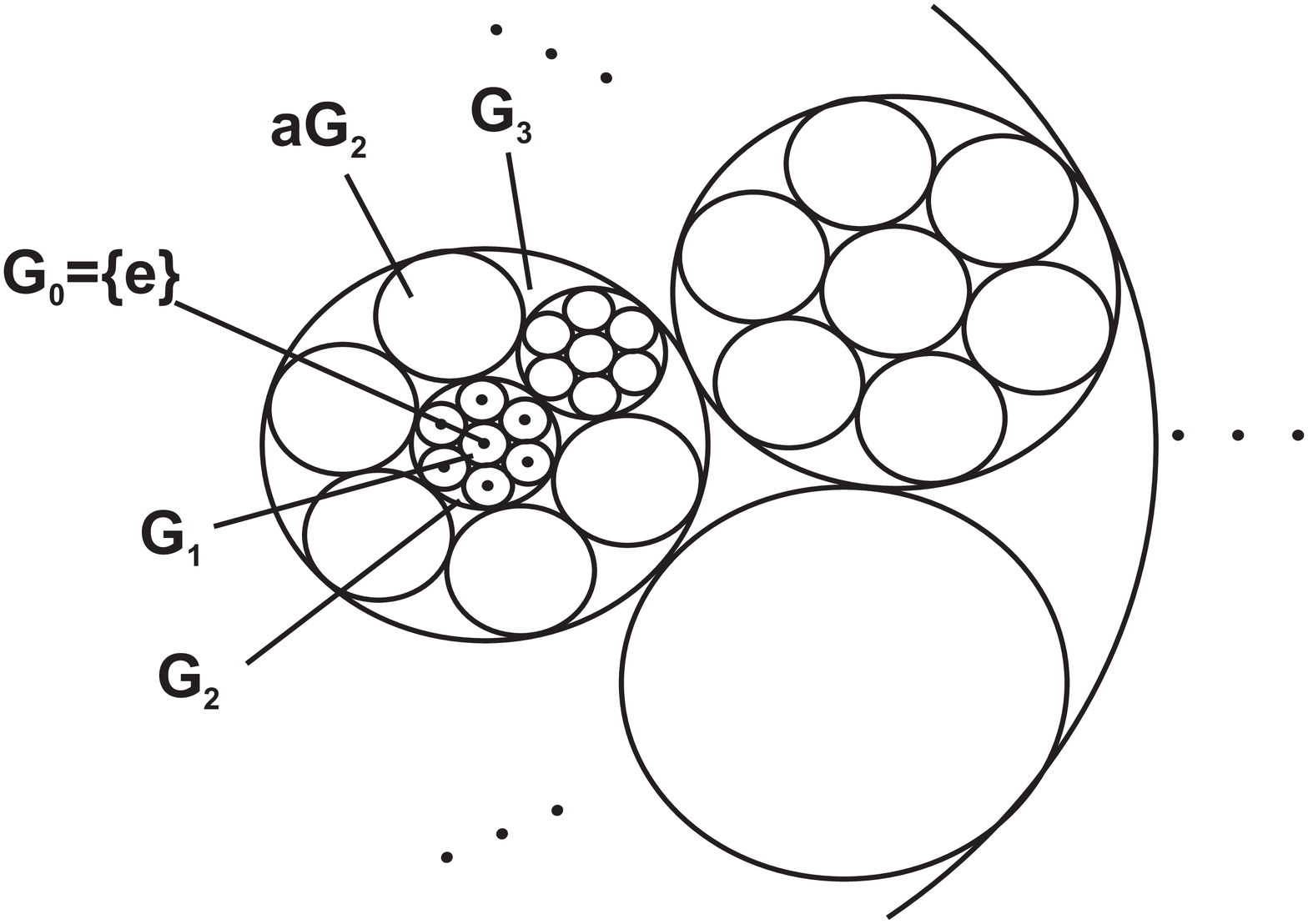}
\caption{The ultra-metric structure of $G$. }
\end{center}
\end{figure}
\begin{pro}\label{pro8.2}
Assume that the following two conditions hold:
\begin{align}\label{equat8.1} 
\lim_{k \to \infty} \frac{\log |G_{k+1}|}{\log |G_k|}=1,
\end{align}
\begin{align} \label{equat8.2}
\sigma(k) \asymp |G_k|^{-\gamma}, \quad \textrm{ for some } \gamma>0.
\end{align}
Then the function $r \to |B_r(a)|$ is of finite order $1/\gamma$.  In particular, for any $\epsilon>0$ there exist $c_1, c_2>0$ such that
\begin{align*}
c_1 r^{\frac{1}{\gamma}}\leq |B_r(a)| \leq c_2 r^{\frac{1}{\gamma}+\epsilon} \qquad \textrm{near} \quad \infty.
\end{align*}
\end{pro}
\begin{proof}
Let $\sigma(k) \leq \frac{1}{1+r} < \sigma(k-1)$. Then $|B_r(a)|=|G_k|$. This yields the following two inequalities:
\begin{align} \label{equat8.3}
|B_r(a)|>\frac{c_1}{(\sigma(k))^{1/\gamma}}\geq c_1(1+r)^{1/\gamma}>c_1 r^{1/\gamma}
\end{align}
and, if $k$ is big enough,
\begin{align} \label{equat8.4}
|B_r(a)|=&|G_k|\leq |G_{k-1}|^{1+\gamma \epsilon}\leq \frac{c_2}{(\sigma(k-1))^{\frac{1}{\gamma}+\epsilon}}<\\
\nonumber <& c_2(1+r)^{\frac{1}{\gamma}+\epsilon}<c_3r^{\frac{1}{\gamma}+\epsilon}.
\end{align} 
Inequalities (\ref{equat8.3}) and (\ref{equat8.4}) imply the result.
\end{proof}
\begin{exa}\textrm{ }\\
1) Let $G=\mathbb{Z}(2)^{(\infty)}$. Assume that $\sigma(k)\asymp q^k$, for some $0<q<1$. Then Conditions (\ref{equat8.1}) and (\ref{equat8.2}) hold and Proposition \ref{pro8.2} applies. Observe however that we can also use Proposition \ref{pro8.1} and Example \ref{ex2}(1.1) to get a more precise result: there exist $c_1,c_2>0$ such that
\begin{align*} 
c_2 r^{\frac{1}{\alpha}}\leq |B_r(a)|\leq c_1r^{\frac{1}{\alpha}}, \qquad \textrm{at} \quad \infty,
\end{align*}
where $\alpha= \log \frac{1}{q}/ \log 2$, and $c_1,c_2>0$ are constants.\\
2) Let $G=S_{\infty}$ be the infinite symmetric group. Assume that $\sigma(k)\asymp (k!)^{-\gamma}$ for some $\gamma>0$. Then Conditions (\ref{equat8.1}) and (\ref{equat8.2}) hold. Hence by Proposition \ref{pro8.2} for any $\epsilon>0$ there exist $c_1,c_2>0$ such that
\begin{align*}
c_2 r^{\frac{1}{\gamma}}\leq |B_r(a)|\leq c_1r^{\frac{1}{\gamma}+\epsilon} \qquad \textrm{at} \quad \infty.
\end{align*}
Note that Proposition \ref{pro8.1} and Example \ref{ex2}(2b) give the same result. 
\end{exa}
Let $X(n)$ be the random walk on $G$ with law $\mu=\mu(c)$ starting from $x=e$. For any $\alpha>0$ \textit{the mean $\alpha$-displacement $M_X(\alpha, n)$ of $X(n)$} is defined as follows:
\begin{align*}
M_X(\alpha, n):=E\left\{\rho(e,X(n))^{\alpha} \right\}.
\end{align*}
There is an obvious way to extend the definition of the relation $f \asymp g$ between
functions depending on several numerical variables.
\begin{pro}
With the above notation, the following properties hold true.\\
$1$. For any fixed $n \geq 1$, $M_X(\alpha, n)<\infty$ if and only if $\,\alpha<1$.\\
$2$. Assume that $\alpha <1$ and that Condition (A) holds. Then $M_X(\alpha, n)$ as a function of $(\alpha, n)$ satisfies
\begin{align*}
M_X(\alpha, n)\asymp \frac{n^{\alpha}}{1-\alpha}.
\end{align*}
\end{pro}
\begin{proof}[Proof of (1):] According to Proposition \ref{pro8.1}, each finite group $G_k$ is a ball $B_{r_k}(e)$ of radius
\begin{align*}
r_k=\frac{1}{\sigma(k-1)}-1.
\end{align*}
We claim that as a function of $(k,n)$
\begin{align}\label{equat8.5}
\mu_n (G \backslash G_k) \asymp \min \left\{ \frac{n}{r_{k+1}},1 \right\}.
\end{align}
Indeed, according to Proposition \ref{proposit3.2}, for any $n \geq 1$ and $k \geq 0$,
\begin{align*}
\mu_n (G \backslash G_k)&=\sum_{l \geq 0} C_l(n) m_l (G \backslash G_k)=\sum_{l>k}C_l(n)\left(1-\frac{|G_k|}{|G_l|}  \right)\asymp \\
&\asymp \sum_{l>k}C_l(n)=1-(1-\sigma(k) )^n\asymp \min \{ n \sigma(k),1\}\asymp\\
&\asymp \min \left\{ \frac{n}{r_{k+1}},1 \right\}.
\end{align*}
Next we write
\begin{align*}
M_X(\alpha,n) &=\int_G \rho(e,y)^{\alpha} d\mu_n (y)=\sum_{k \geq 0} \int_{G_{k+1}\backslash G_k} \rho(e,y)^{\alpha} d\mu_n (y)=\\
&=\sum_{k \geq 0} r_{k+1}^{\alpha} \mu_n(G_{k+1}\backslash G_k)=\\
&=\sum_{k \geq 0} r_{k+1}^{\alpha}\left( \mu_n (G \backslash G_k) - \mu_n (G\backslash G_{k+1})\right)=\\
&=r_{1}^{\alpha} \mu_n (G \backslash G_0)+ \sum_{k \geq 1} \left( r_{k+1}^{\alpha}- r_{k}^{\alpha} \right) \mu_n (G \backslash G_k)=\\
&=r_{1}^{\alpha}(1-p(n))+R(\alpha, n).
\end{align*}
Using the equivalence relation (\ref{equat8.5}) we obtain that as a function of~$(\alpha,n)$,
\begin{align}\label{equat8.6}
R(\alpha, n) \asymp \sum_{k \geq 1} \left( r_{k+1}^{\alpha}- r_{k}^{\alpha} \right)\min \left\{ \frac{n}{r_{k+1}},1 \right\}.
\end{align}
Let $0< \alpha <1$, then
\begin{align*}
\sum_{k >> 1} \left( r_{k+1}^{\alpha}- r_{k}^{\alpha} \right)\frac{1}{r_{k+1}}&=\alpha \sum_{k >> 1} \frac{1}{r_{k+1}}
\int_{r_k}^{r_{k+1}} r^{\alpha -1} dr \leq \\ 
& \leq \alpha \sum_{k >> 1} \int_{r_k}^{r_{k+1}} r^{\alpha -2} dr \leq \alpha \int_{r_1}^{\infty} r^{\alpha -2} dr<\infty.
\end{align*}
Hence, for any fixed $n \geq 1$ and $0< \alpha <1$,
\begin{align*}
M_X(\alpha,n) \leq r_1^{\alpha} +R(\alpha, n)<\infty.
\end{align*}
For $\alpha \geq 1$ write
\begin{align*}
\sum_{k >> 1} \left( r_{k+1}^{\alpha}- r_{k}^{\alpha} \right)\frac{1}{r_{k+1}}&=\alpha \sum_{k >> 1} \frac{1}{r_{k+1}}
\int_{r_k}^{r_{k+1}} r^{\alpha -1} dr \geq \\ 
& \geq \sum_{k >> 1} \frac{1}{r_{k+1}}
\int_{r_k}^{r_{k+1}} dr = \sum_{k >> 1} \left( 1- \frac{r_k}{r_{k+1}}\right).
\end{align*}
Assume that the series $\sum_k (1- r_k/ r_{k+1})$ converges. Then, $r_k/ r_{k+1} \to 1$ at $\infty$.
It follows that for any $\epsilon >0$ there exists $k_0>1$ such that for all $k \geq k_0$,
\begin{align*}
r_k < r_{k+1} \leq (1+\epsilon) r_k.
\end{align*}
Hence, for $\epsilon=1$,
\begin{align*}
\sum_{k \geq k_0} \left(1-\frac{r_k}{r_{k+1}} \right) &= \sum_{k \geq k_0} \frac{r_{k+1}-r_k}{r_{k+1}}\geq \frac{1}{2} \sum_{k \geq k_0}  \frac{r_{k+1}-r_k}{r_{k}}=\\
&= \frac{1}{2} \sum_{k \geq k_0} \left(\frac{r_{k+1}}{r_{k}}-1 \right) \geq  \frac{1}{2} \sum_{k \geq k_0} \log \frac{r_{k+1}}{r_{k}}=+\infty.
\end{align*}
This is a contradiction. Hence the series $\sum_k (1- r_k/ r_{k+1})$ diverges and we conclude that for any $n \geq 1$ and $\alpha \geq 1$,
\begin{align*}
M_X(\alpha,n) \geq R(\alpha, n)=\infty.
\end{align*}
\textit{Proof of (2):} Let $0< \alpha <1$. Evidently for all $n \geq 1$,
\begin{align*}
\max \left\{ r_1^{\alpha}(1-p(1)), R(\alpha,n)\right\}<M_X(\alpha,n) <r_1^{\alpha}+ R(\alpha, n).
\end{align*}
Let $c_1,c_2>0$ be constants which justify the equivalence relation (\ref{equat8.6}). Choose $k_0 \geq 1$ such that $r_{k_0+1}\leq n \leq r_{k_0+2}$ and write
\begin{align*}
R(\alpha, n) & \leq c_1 \left( \sum_{k\leq k_0}\left( r_{k+1}^{\alpha}- r_{k}^{\alpha} \right)+n \sum_{k \geq k_0+1} \left( r_{k+1}^{\alpha}- r_{k}^{\alpha} \right)\frac{1}{r_{k+1}} \right)=\\
&= c_1 \left(  r_{k_0+1}^{\alpha}- r_{1}^{\alpha} +n \sum_{k \geq k_0+1} \left( r_{k+1}^{\alpha}- r_{k}^{\alpha} \right)\frac{1}{r_{k+1}} \right)\leq\\
&\leq c_1 \left( n^{\alpha} +\alpha n \int_{k_0+1}^{\infty} r^{\alpha -2} dr \right)=\\
&= c_1 \left( n^{\alpha} +\frac{\alpha n}{1-\alpha} \left( \frac{1}{r_{k_0+1}}\right)^{1-\alpha}\right).
\end{align*}
Condition (A) implies that for some $\lambda>0$,
\begin{align*}
r_{k+1}\leq (1+\lambda ) r_k, \qquad k\geq 1.
\end{align*}
Hence, 
\begin{align*}
R(\alpha, n)&\leq c_1 \left( n^{\alpha} +\frac{\alpha n(1+\lambda)^{1-\alpha}}{1-\alpha} \left( \frac{1}{r_{k_0+2}}\right)^{1-\alpha}\right)\leq \\
&\leq c_1\left( (1-\alpha )+\alpha (1+\lambda )^{1-\alpha}\right) \cdot \frac{n^{\alpha}}{1-\alpha } \leq \\
&\leq c_1 (1+\lambda) \frac{n^{\alpha}}{1-\alpha }.
\end{align*}
This inequality yields the upper bound
\begin{align*}
M_X(\alpha,n) \textrm{\Huge{/}} \frac{n^{\alpha}}{1-\alpha } &\leq c_1(1+\lambda ) +(1-\alpha ) \left( \frac{r_1}{n}\right)^{\alpha} \leq\\
&\leq c_1(1+\lambda )+ \max (r_1,1).
\end{align*}
On the other hand, by (\ref{equat8.6}), for $n^{\alpha} \geq 2(1+\lambda) r_1^{\alpha}$,
\begin{align*}
R(\alpha, n)& \geq c_2 \sum_{k\geq 1}\left( r_{k+1}^{\alpha}- r_{k}^{\alpha} \right)\min \left\{\frac{n}{r_{k+1}},1\right\}=\\
&=c_2 \left( \sum_{k\leq k_0}\left( r_{k+1}^{\alpha}- r_{k}^{\alpha} \right)+n \sum_{k \geq k_0+1} \left( r_{k+1}^{\alpha}- r_{k}^{\alpha} \right)\frac{1}{r_{k+1}} \right)=\\
&= c_2 \left( r_{k_0+1}^{\alpha} -r_1^{\alpha} + n \alpha \sum_{k \geq k_0+1} \frac{1}{r_{k+1}} \int_{r_k}^{r_{k+1}} r^{\alpha -1} dr \right) \geq\\
&\geq c_2 \left( \frac{r_{k_0+2}^{\alpha}}{1+\lambda} -r_1^{\alpha} + \frac{ n \alpha}{1+\lambda} \sum_{k \geq k_0+1} \frac{1}{r_{k}} \int_{r_k}^{r_{k+1}} r^{\alpha -1} dr \right) \geq \\ 
&\geq  c_2 \left( \frac{n^{\alpha}}{1+\lambda} -r_1^{\alpha} + \frac{ n \alpha}{1+\lambda} \sum_{k \geq k_0+1} \int_{r_k}^{r_{k+1}} r^{\alpha -2} dr \right)=\\
&=  c_2 \left( \frac{n^{\alpha}}{1+\lambda} -r_1^{\alpha} + \frac{ n \alpha}{1+\lambda} \int_{r_{k_0+1}}^{\infty} r^{\alpha -2} dr\right)=\\
&=c_2 \left( \frac{n^{\alpha}}{1+\lambda} -r_1^{\alpha} + \frac{n \alpha}{(1+\lambda )(1-\alpha )}  \left( \frac{1}{r_{k_0+1}} \right)^{1-\alpha}\right) \geq\\
&\geq c_2 \left( \frac{n^{\alpha}}{1+\lambda} -r_1^{\alpha} + \frac{n \alpha}{(1+\lambda )(1-\alpha )} \cdot \frac{1}{n^{1-\alpha}} \right)=
\end{align*}
\begin{align*}
&=  c_2 \left( \frac{n^{\alpha}}{(1+\lambda )(1-\alpha )} -  r_1^{\alpha} \right)\geq\\
&\geq \frac{c_2 }{2(1+\lambda )}\cdot\frac{n^{\alpha}}{1-\alpha}.
\end{align*}
Hence, for $n^{\alpha} \geq 2(1+\lambda) r_1^{\alpha}$ and $0<\alpha <1$, we obtain
\begin{align*} 
M_X(\alpha,n) \textrm{\Huge{/}} \frac{n^{\alpha}}{1-\alpha } \geq R(\alpha,n) \textrm{\Huge{/}} \frac{n^{\alpha}}{1-\alpha } \geq \frac{c_2 }{2(1+\lambda )}.
\end{align*}
If $n^{\alpha} < 2(1+\lambda) r_1^{\alpha}$ and $0<\alpha \leq \frac{1}{2}$,
\begin{align*} 
M_X(\alpha,n) \textrm{\Huge{/}} \frac{n^{\alpha}}{1-\alpha } &\geq \frac{r_1^{\alpha}(1-p(1))(1-\alpha )}{n^{\alpha}} \geq \frac{r_1^{\alpha}(1-p(1))(1-\alpha )}{2(1+\lambda) r_1^{\alpha}}=\\
&= \frac{(1-p(1))(1-\alpha )}{2(1+\lambda)}\geq \frac{(1-p(1))}{4(1+\lambda)}.
\end{align*}
If $\frac{1}{2}< \alpha <1$ we repeat our computations from above but without the term $\sum_{k \leq k_0}(r_{k+1}^{\alpha} - r_k^{\alpha})\min \left\{ \frac{n}{r_{k+1}},1\right\}$ and obtain
\begin{align*} 
R(\alpha, n) &\geq c_2 \sum_{k \geq k_0+1} \left( r_{k+1}^{\alpha} - r_k^{\alpha}\right) \min \left\{\frac{n}{r_{k+1}},1\right\} \geq \frac{c_2 \alpha}{1+\lambda}\cdot \frac{n^{\alpha}}{1-\alpha} >\\
&> \frac{c_2}{2(1+\lambda)}\cdot \frac{n^{\alpha}}{1-\alpha}.
\end{align*}
In particular, for $n^{\alpha} < 2(1+\lambda) r_1^{\alpha}$ and $\frac{1}{2}< \alpha <1$,
\begin{align*} 
M_X(\alpha,n) \textrm{\Huge{/}} \frac{n^{\alpha}}{1-\alpha } \geq R(\alpha,n) \textrm{\Huge{/}} \frac{n^{\alpha}}{1-\alpha } > \frac{c_2 }{2(1+\lambda )}.
\end{align*}
The three inequalities obtained above yield the desired low bound.
The proof is finished.
\end{proof}

Let $(\mu_t)_{t>0}$ be a weakly continuous convolution semigroup of probability measures such that $\mu_t|_{t=1}=\mu(c)$, see Proposition \ref{proposit3.2}. Write
\begin{align*} 
P_tf(x)=f*\mu_t(x)=\int_G f(y)h(t;x,y)dm(y),
\end{align*}
and call $(P_t)_{t>0}$ \textit{the heat semigroup} and $h(t;x,y)$ \textit{the heat kernel} associated to the measure $\mu(c)$.
\begin{pro}\label{pro8.4}
Let $\mu$, $h$ be as above and $\rho:=\rho(x,y)$, then 
\begin{align} \label{equat8.7}
h(t;x,y)=t\int_0^{\frac{1}{1+\rho}}N(\lambda)(1-\lambda)^{t-1}d\lambda.
\end{align}
\end{pro}
\begin{proof}
Put $x^{-1}y=z$, then
\begin{align*} 
h(t;x,y)=\mu_t(\{x^{-1}y\})=\mu_t(\{z\}).
\end{align*}
By Proposition \ref{pro6}, for $z\in G_k \backslash G_{k-1}$,
\begin{align*} 
\mu_t(z)=& \sum_{ n \geq k} \frac{1}{|G_n|} \left[ (1-\sigma (n) )^t -(1-\sigma (n-1))^t \right] =\\
=& \sum_{n\geq k} N(\sigma(n) ) \left[ (1-\sigma(n))^t -(1-\sigma(n-1))^t\right]=\\
=& \sum_{n\geq k} N(\sigma(n) )t\int_{\sigma(n)}^{\sigma(n-1)} (1-\lambda)^{t-1}d\lambda =\\
=& t\int_0^{\sigma(k-1)} N(\lambda) (1-\lambda)^{t-1}d\lambda= \\
=& t\int_0^{\frac{1}{1+|z|_{\sigma}}} N(\lambda) (1-\lambda)^{t-1}d\lambda=  t\int_0^{\frac{1}{1+\rho}} N(\lambda) (1-\lambda)^{t-1}d\lambda.
\end{align*}
The proof is finished.
\end{proof}
Recall that according to (\ref{equat7.3}) the return probability $p(t)$ and the spectral distribution $N(\lambda)$ are related by the formula
\begin{align*}
p(t)=h(t;e,e)=t\int_0^1 N(\lambda)(1-\lambda)^{t-1}d\lambda \stackrel{d}{\asymp}t\int_0^1 N(\lambda)e^{-t\lambda}d\lambda.
\end{align*}
It is easy to see that we always have
\begin{align}  \label{equat8.8}
p(t)\geq \frac{(1-\sigma(0))^2}{2e^2} N\left( \frac{1}{t}\right), \qquad t\geq 1.
\end{align}
Indeed, for $t \geq 2$,
\begin{align*} 
p(t)\geq t \int_{\frac{1}{t}}^1 N(\lambda) (1-\lambda )^{t-1} d\lambda \geq N\left( \frac{1}{t}\right) \left( 1-\frac{1}{t}\right)^t \geq e^{-2} N\left( \frac{1}{t}\right)
\end{align*}
and, for $1\leq t<2$,
\begin{align*} 
p(t)\geq \int_{\sigma(0)}^1 N(\lambda)(1-\lambda) d\lambda =\frac{1}{2} (1-\sigma(0))^2.
\end{align*}
In some cases however
\begin{align*} 
p(t)\stackrel{d}{\asymp} N\left( \frac{1}{t}\right) \qquad \textrm{at} \quad \infty.
\end{align*}
This happens if for instance, $p(t)\asymp t^{-\beta}$ at $\infty$ for some $\beta>0$, equivalently, $N(\lambda)\asymp \lambda^{\beta}$ at $0$. See \cite[Thm. 1.7.1.]{BGT}, where classical Karamata arguments justify this statement. 
\begin{pro}\label{pro8.5}
Let $h$, $N$ and $\rho$ be as above. The following relation holds:
\begin{align}\label{equat8.9} 
h(t;x,y)\stackrel{d}{\asymp} t\int_0^{\frac{1}{1+\rho}} N(\lambda)e^{-t\lambda}d\lambda.
\end{align}
\end{pro}
\begin{proof}
When $\rho=\rho(x,y)=0$, (\ref{equat8.9}) reduces to (\ref{equat7.3}). Let $\rho>0$, then $\rho\geq \frac{1}{\sigma(0)}-1$ and $\frac{1}{1+\rho}\leq \sigma(0) <1$. Therefore we can write
\begin{align*} 
h(t;x,y)= t\int_0^{\frac{1}{1+\rho}} N(\lambda)(1-\lambda)^{t}\frac{d\lambda}{1-\lambda} \leq \frac{t}{1-\sigma(0)}\int_0^{\frac{1}{1+\rho}} N(\lambda)e^{-t\lambda}d\lambda.
\end{align*}
On the other hand,
\begin{align*} 
h(t;x,y)\geq t\int_0^{\frac{1}{1+\rho}} N(\lambda)(1-\lambda)^{t}d\lambda \geq t\int_0^{\frac{1}{1+\rho}} N(\lambda)e^{-\delta t\lambda}d\lambda,
\end{align*}
where one can choose $\delta=\frac{1}{1-\sigma(0)}$.
\end{proof}
There is an obvious way to define $f  \stackrel{d}{\leq} g$ so that
\begin{align*} 
 f  \stackrel{d}{\leq} g\quad \textrm{and} \quad  g  \stackrel{d}{\leq} f \iff  f  \stackrel{d}{\asymp} g.
\end{align*}
\begin{pro}\label{pro8.6} Assume that $t+\rho \to \infty$. 
Then the following conditions hold true.\\
$1$. If $\,t/(1+\rho)\leq 1$, then
\begin{align} \label{equat8.10}
h(t;x,y)\stackrel{d}{\asymp}\frac{t}{1+\rho}N\left(\frac{1}{1+\rho}\right).
\end{align}
$2$. If $\,t/(1+\rho) >1$, then
\begin{align} \label{equat8.11}
N\left( \frac{1}{t}\right) \stackrel{d}{\leq} h(t;x,y) \leq p(t).
\end{align}
In particular, for all $x,y \in G$ and $t\geq 1$,
\begin{align}  \label{equat8.12}
h(t;x,y) \stackrel{d}{\geq} \frac{t}{t+\rho}N\left(\frac{1}{t+\rho}\right)
\end{align}
and
\begin{align}\label{equat8.13}
h(t;x,y) \stackrel{d}{\leq} \frac{t}{t+\rho}p(t+\rho).
\end{align}
\end{pro}
\begin{cor}\label{cor8.7} Assume that $p(t)\stackrel{d}{\asymp}N\left(\frac{1}{t}\right)$. Then for all $x,y\in G$ and $t\geq 1$,
\begin{align} 
h(t;x,y) \stackrel{d}{\asymp} \frac{t}{t+\rho}N\left(\frac{1}{t+\rho}\right).
\end{align}
\end{cor}
\begin{proof}[Proof of Proposition \ref{pro8.6}(1):] The relation (\ref{equat8.9})
implies the following two inequalities:
\begin{align*} 
h(t;x,y)\stackrel{d}{\leq} N\left(\frac{1}{1+\rho}\right)\left(1-e^{-\frac{t}{1+\rho}}\right) \leq \frac{t}{1+\rho}N\left(\frac{1}{1+\rho}\right)
\end{align*}
and
\begin{align*} 
h(t;x,y) &\stackrel{d}{\geq} t \int_{\frac{1}{2(1+\rho)}}^{\frac{1}{1+\rho}} N(\lambda)e^{-t\lambda}d\lambda \geq \frac{t}{2(1+\rho)}N\left(\frac{1}{2(1+\rho)}\right)e^{-\frac{t}{1+\rho}} \geq \\
& \geq \frac{t}{2e(1+\rho)}N\left(\frac{1}{2(1+\rho)}\right).
\end{align*}
These inequalities show that the first statement holds true.\\
\textit{ Proof of Proposition \ref{pro8.6}(2):}
Observe that for all $x,y\in G$ and $t\geq 1$,
\begin{align*} 
h(t;x,y)\leq h(t;e,e)=p(t).
\end{align*}
On the other hand, since $\frac{t}{1+\rho}>1$, we must have $\frac{1}{1+\rho}>\frac{1}{t}$. Hence, for such $t> 1$, $x,y\in G$ and for some $\delta>1$ we obtain
\begin{align*} 
h(t;x,y) &{\geq} t \int_{0}^{\frac{1}{t}} N(\lambda)e^{-t\delta \lambda}d\lambda \geq e^{-\delta}t \int_{0}^{\frac{1}{t}} N(\lambda ) d\lambda \geq \\
&\geq e^{-\delta}t \int_{\frac{1}{2t}}^{\frac{1}{t}} N(\lambda ) d\lambda \geq \frac{1}{2} e^{-\delta} N\left( \frac{1}{2t}\right).
\end{align*}
Finally, (\ref{equat8.12}) and (\ref{equat8.13}) follow from (\ref{equat8.10}) and (\ref{equat8.11}). Indeed, depending on whether the quantity $\frac{t}{1+\rho}$ is less or equal than one, or greater than one, we have the following inequalities.

(1) If $\frac{t}{1+\rho}\leq 1$, then
\begin{align*} 
h(t;x,y) \stackrel{d}{\geq}\frac{t}{1+\rho}N\left(\frac{1}{1+\rho}\right) \geq \frac{t}{t+\rho}N\left(\frac{1}{t+\rho}\right). 
\end{align*}
By assumption $t\leq 1+\rho$, hence
\begin{align*} 
t+\rho \leq 2(1+\rho).
\end{align*}
It follows that
\begin{align*} 
\frac{t}{t+\rho} \geq \frac{1}{2}\frac{t}{1+\rho},
\end{align*}
consequently, by (\ref{equat8.8})
\begin{align*} 
\frac{t}{t+\rho} p(t+\rho) &\geq \frac{t}{2(1+\rho)} p(2(1+\rho))\geq\\ &\geq \frac{(1-\sigma(0))^2}{4e^2}\frac{t}{1+\rho}N\left(\frac{1}{2(1+\rho)}\right) \stackrel{d}{\asymp} h(t;x,y).
\end{align*}

(2) If $\frac{t}{1+\rho}> 1$, then
\begin{align*} 
h(t;x,y) \stackrel{d}{\geq}N\left(\frac{1}{t}\right) \geq \frac{t}{t+\rho}N\left(\frac{1}{t}\right)\geq \frac{t}{t+\rho}N\left(\frac{1}{t+\rho}\right) 
\end{align*}
and
\begin{align*} 
\frac{t}{t+\rho} p(t+\rho) \geq \frac{1+\rho}{1+2\rho} p(2t) \geq \frac{1}{2} p(2t) \stackrel{d}{\asymp}p(t) \geq h(t;x,y).
\end{align*}
The proof is finished.
\end{proof}
\begin{exa}\label{ex8.9}\textrm{ }\\
1) Let $G=\mathbb{Z}(2)^{(\infty)}$ and $\sigma(k) \asymp q^k$, $0<q<1$. According to Examples \ref{ex2}(1.1) and \ref{ex3}(1.1)
\begin{align*} 
N(\lambda)\asymp \lambda^{\frac{1}{\alpha}} \qquad \textrm{at }\quad 0,
\end{align*}
and 
\begin{align*} 
p(t)\asymp t^{-\frac{1}{\alpha}} \qquad \textrm{at } \quad \infty,
\end{align*}
where $\alpha=\log \frac{1}{q}/ \log 2$. Thus Corollary \ref{cor8.7} applies and we obtain
\begin{align}\label{eq8.15}
h(t;x,y)\asymp \frac{t}{(t+\rho(x,y))^{1+\frac{1}{\alpha}}}. 
\end{align}
2) Let $G=S_{\infty}$ and $\sigma(k)\asymp (k!)^{-\gamma}$, for some $\gamma>0$. According to Example \ref{ex2}(2b), for any $\epsilon>0$, there exist $c_1, c_2>0$ such that
\begin{align*} 
c_2 \lambda^{\frac{1}{\gamma}+\epsilon} \leq N(\lambda) \leq c_1 \lambda^{\frac{1}{\gamma}} \qquad \textrm{at} \quad 0,
\end{align*}
and similarly, for some $c_3, c_4>0$ 
\begin{align*} 
c_3 t^{-\left(\frac{1}{\gamma}+\epsilon\right)} \leq p(t) \leq c_4 t^{-\frac{1}{\gamma}} \qquad \textrm{at} \quad \infty. 
\end{align*}
Applying Proposition \ref{pro8.6} we obtain
\begin{align}\label{eq8.16}
h(t;x,y)\leq \frac{c_5t}{(t+\rho(x,y))^{1+\frac{1}{\gamma}}}. 
\end{align}
and
\begin{align}\label{eq8.17}
h(t;x,y)\geq \frac{c_6t}{(t+\rho(x,y))^{1+\frac{1}{\gamma}+\epsilon}}. 
\end{align}
for some $c_5, c_6>0$ and all $x,y \in G$, $t \geq 1$.
\end{exa} 
\begin{rem}\label{rem8.10}
$1$. The heat kernel bounds (\ref{equat8.12}) and (\ref{equat8.13}) given in Proposition \ref{pro8.6} are optimal. Indeed, if $t\geq 1$ is bounded and $\rho=\rho(x,y)\to \infty$, (\ref{equat8.10}) yields
\begin{align*}
h(t;x,y) \stackrel{d}{\asymp}\frac{t}{t+\rho}N\left(\frac{1}{t+\rho}\right).
\end{align*}
On the other hand, if $\rho=\rho(x,y)$ is bounded and $t\to \infty$,
\begin{align*}
h(t;x,y) \leq p(t),
\end{align*}
and by (\ref{equat8.7}), 
\begin{align*}
h(t;x,y) = p(t)-t \int_{\frac{1}{1+\rho}}^1 N(\lambda)(1-\lambda)^{t-1}d\lambda\geq p(t)-\left(1-\frac{1}{1+\rho}\right)^t\sim p(t).
\end{align*}
Thus, in this case,
\begin{align*}
h(t;x,y) \sim p(t) \stackrel{d}{\asymp} \frac{t}{t+\rho}p(t+\rho).
\end{align*}
$2$. Let $h^{\beta}(t;x,y)$ be the heat kernel associated with the symmetric stable process in $\mathbb{R}^d$ of index $0<\beta < 2$. According to \cite{Blu} (see also \cite{Ben2})
\begin{align*}
h^{\beta}(t;x,y)\asymp \frac{t}{\left( t^{1/\beta} +|x-y|   \right)^{\beta +d}}.
\end{align*}
The relations (\ref{eq8.15}) and (\ref{eq8.16})-(\ref{eq8.17}) show that if $\sigma(k)\asymp |G_k|^{-\gamma}$ the heat kernel on the group $\mathbb{Z}(2)^{(\infty)}$ (resp. $S_{\infty}$) has a shape similar to that of the $1$-stable law of ``dimension'' $d=\frac{1}{\alpha}$, (resp. $d=\frac{1}{\gamma} + \epsilon$).
\end{rem}

\begin{xack}
This paper was started at Bielefeld University (SFB 701). The authors express their gratitude to A.~Grigor'yan for his kind invitation and encouragement. Part of the paper was written
during stays at the Laboratoire Poncelet in Moscow (UMI 2615) and at the
Steklov Mathematical Institute. We are grateful to M. Tsfasman and to A. Shiryaev for their invitations. We thank A.~Erschler, Y.~Guivarc'h, V.~Kaimanovich, L.~Saloff-Coste, R.~Shah, A.~Vershik, and W.~Woess, for fruitful discussions and valuable comments. 
\end{xack}

\end{document}